\documentclass[11pt]{amsart}
\usepackage{amssymb, latexsym, mathrsfs, color, tikz, multirow,mathtools,xcolor,enumerate,caption, subcaption, graphicx,xcolor,comment,multicol}
\usepackage{graphics}
\graphicspath{ {./Diagrams/} }

\usepackage{xspace}
\input{xy}
\xyoption{all}

\usepackage[colorlinks=true, pdfstartview=FitV,
 linkcolor=blue,citecolor=blue,urlcolor=blue]{hyperref}
\usetikzlibrary{shapes,snakes}

\usepackage{ytableau} %Young Tableau

\numberwithin{equation}{section}
\theoremstyle{plain}
\newtheorem{theorem}[equation]{Theorem}
\newtheorem{proposition}[equation]{Proposition}
\newtheorem{corollary}[equation]{Corollary}
\newtheorem{lemma}[equation]{Lemma}

\newtheorem {problem}{Problem}

\theoremstyle{definition}
\newtheorem{Def}[equation]{Definition}

\newtheorem{example}[equation]{Example}
\newtheorem{Rmk}[equation]{Remark}

\def\d{\delta}

\def\l{\lambda}

\def\cF{\mathcal{F}}

\def\cF{\mathcal{F}}

\DeclareMathOperator{\st}{\operatorname{\mathtt{ST}}\xspace}
\DeclareMathOperator\ts{\operatorname{\mathtt{TS}}\xspace}
\DeclareMathOperator\sd{\operatorname{\mathtt{SD}}\xspace}
\DeclareMathOperator\ds{\operatorname{\mathtt{DS}}\xspace}
\DeclareMathOperator\dt{\operatorname{\mathtt{DT}}\xspace}
\DeclareMathOperator\td{\operatorname{\mathtt{TD}}\xspace}

\newcommand{\U}{\textsf{U}}
\newcommand{\D}{\textsf{D}}

\def\Skyt{\operatorname{Skyt}}
\def\Nom{\operatorname{Nom}}
\def\Dyck{\operatorname{Dyck}}
\def\Tri{\operatorname{Tri}}
\def\Dis{\operatorname{Dis}}

\definecolor{applegreen}{rgb}{0.55,0.71,0.0}
\newcommand{\editone}[1]{\textcolor{black}{#1}}

\newcommand{\edittwo}[1]{\textcolor{black}{#1}}

\definecolor{darkpink}{RGB}{245,75,200}

\makeatletter

\def\showlineno{line \the\inputlineno}

\definecolor{linenogrey}{RGB}{150,150,150}

\title{Combinatorics for certain Skew Tableaux, Dyck Paths, Triangulations, and Dissections}
\author{Su Ji Hong}
\address{Department of Mathematics\\
         Yale University\\
\url{%https://www.math.unl.edu/~shong3/
}}
\email{suji.hong@yale.edu}

\author{George D. Nasr}
\thanks{Nasr is partially supported by the National Science Foundation under  grant DMS-2053243 (FRG)}
\address{Department of Mathematics\\
         University of Oregon\\
\url{https://sites.google.com/view/george-d-nasr-math}}
\email{gdnasr@uoregon.edu}

\begin{document}
\begin{abstract}
 We present combinatorial bijections and identities between certain skew Young tableaux, Dyck paths, triangulations, and dissections.  
\end{abstract}
\maketitle

\section{Introduction}

Bijections between Catalan objects are well understood. For instance, see \cite{StanleyBook}. %
There are many generalizations of these objects and the bijections between them.
In \cite{S96}, Stanley provides a bijection between certain standard Young tableaux and dissections of a polygon. 
In \cite[Proposition 2.3]{GMTW20}, the authors provide a bijection between the same tableaux and certain Dyck paths. 
Meanwhile, various papers consider a certain collection of skew Young tableaux---which may be seen as a generalization of the aforementioned tableaux---which are used to compute formulas for the ordinary and equivariant Kazhdan--Lusztig polynomial for uniform, sparse paving, and paving matroids \cite{LNRPre20,LNR20,fnv2021,GXY21,KNPV21}. \footnote{Kazhdan-Lusztig polynomials for matroids were first defined in \cite{EPW16}.}

\editone{The primary goal of this paper is to generalize the bijection in \cite[Proposition 2.3]{GMTW20}, so that it involves the skew tableaux mentioned above, while simultaneously including bijections involving certain triangulations.
As a result of these bijections}, properties about the skew tableaux will have implications for the Dyck paths and triangulation objects of interest. 
Motivated by our findings, we then find a combinatorial bijection between the dissections in \cite{S96} and our triangulations. 

In the next section, we will define relevant terminology for skew Young tableaux in subsection \ref{subsec:syt_back}, Dyck paths in subsection \ref{subsec:dyckback}, and then both dissections and triangulations in subsection \ref{subsec:distriback}. 
Then in subsection \ref{subsec:mainback}, we discuss the main results and findings of this paper in detail. 
{In sections \ref{sec:combinatorial_bijections} and \ref{sec:enumeration}, we provide the definitions for the maps involved in the main results.}
%The reader familiar with any of the previously mentioned combinatorial objects will still benefit from reading the corresponding sections, as we define some notation specific to this paper in each section.

\subsubsection*{Acknowledgements:} The authors would like to thank Kyungyong Lee for his helpful input on this paper.

\section{Background and Main Results}
\label{sec:background}
\subsection{Skew Young Tableaux and Nomincreasing Partitions}\label{subsec:syt_back}

\begin{Def}
Let $\lambda_1\geq \lambda_2\geq \cdots \geq \lambda_k$ be positive integers. We say that $\lambda=[\lambda_1,\lambda_2,\dots,\lambda_k]$ is a \textit{partition} of $n$ if $\lambda_1+\cdots +\lambda_k=n$. 
The \textit{Young diagram of shape} $\lambda$ is represented by boxes that are left justified so that the $i$th row has $\lambda_i$ boxes.
 A \textit{standard Young tableau} is achieved by filling the boxes with the numbers so that 
\begin{itemize}
\item each row strictly increases from left to right;
\item each column increases from top to bottom; and
\item if there are $n$ boxes, only the numbers 1 through $n$ are used. 
\end{itemize}
\end{Def} See Figure \ref{fig:tableaux} below for an example of a Young diagram and standard Young tableaux.

\begin{figure}[h]
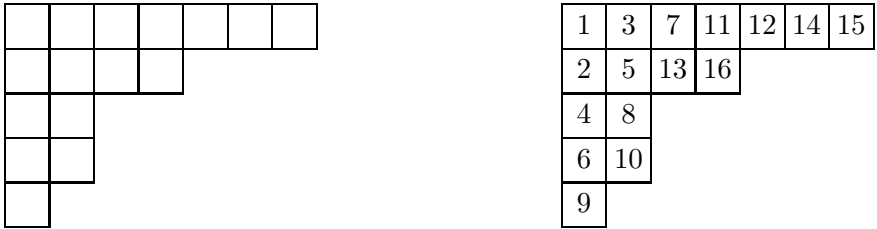

\begin{center}
\begin{minipage}{.45\textwidth}
\ydiagram{7,4,2,2,1}
\end{minipage}
\begin{minipage}{.45\textwidth}
\begin{ytableau}
       1 & 3 & 7 & 11 & 12  & 14 & 15\\
       2 & 5 & 13 & 16 \\
       4 & 8  \\
       6 & 10 \\
       9 
\end{ytableau}
\end{minipage}
\end{center}
\caption{The Young diagram and a standard Young tableaux of shape $[7,4,2,2,1]$}\label{fig:tableaux}
\end{figure}

\begin{Def}
Given partitions $\mu=[\mu_1,\dots, \mu_\ell]$ and $\lambda=[\lambda_1,\dots, \lambda_k]$ so that $\mu_i\leq \lambda_i$ for all $i$, the  \textit{skew Young diagram} $\lambda\setminus \mu$ is the set of squares from the diagram for $\lambda$ that are not in the diagram for $\mu$. 
%\suji{In the definition of skew young diagram, shouldn't $\mu_1 +\cdots \mu_i \leq \lambda_1+\cdots+\lambda_i$? \GDN{Definition updated.}}
As before, we define a \textit{skew Young tableau} to be a skew Young diagram filled with numbers following the same rules described for standard Young tableau. 
\end{Def}
See Figure \ref{fig:skew_tableaux} for an example of a skew Young tableaux.

\begin{figure}[h]
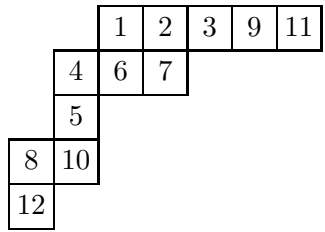

\begin{ytableau}
       \none & \none & 1 & 2& 3  & 9 & 11\\
       \none & 4 & 6 & 7 \\
       \none & 5  \\
       8 & 10 \\
       12 
\end{ytableau}
\caption{A skew Young tableaux of shape $\lambda\setminus \mu$ where $\lambda=[7,4,2,2,1]$ and $\mu=[2,1,1]$.}\label{fig:skew_tableaux}
\end{figure}

The authors in \cite{LNR20} introduce the notation $\Skyt(a,i,b)$ to denote the skew Young tableaux of shape $[(i+1)^b,1^{a-2}]/[(i-1)^{b-2}]$, where we write $x^t$ to denote $x,x,\dots, x$, where $x$ is written $t$ times. 
These are precisely the skew tableaux we discussed in the introduction.
The diagram for the tableaux in $\Skyt(a,i,b)$ is shown in Figure \ref{fig:skyt_syt_diagrams}. 

\begin{center}
\begin{figure}[h]
 \begin{tikzpicture}[scale=0.5, line width=0.9pt]
            \draw (-1,0) grid (0,6);
            \draw[decoration={brace,raise=7pt},decorate]
            (-1,0) -- node[left=7pt] {$a$} (-1,6);
            \draw (0,4) grid (4,6);
            \draw[decoration={brace,mirror, raise=4pt},decorate]
            (0.1,4) -- node[below=7pt] {$i$} (5,4);
            \draw (4,4) grid (5,9);
            \draw[decoration={brace,mirror, raise=5pt},decorate]
            (5,4) -- node[right=7pt] {$b$} (5,9); 
        \end{tikzpicture}
\caption{The diagram for $\Skyt(a,i,b)$.}\label{fig:skyt_syt_diagrams}
\end{figure}
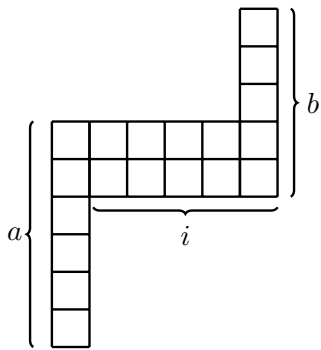
\end{center}

\subsection{Dyck Paths}\label{subsec:dyckback}

\begin{Def}
A \textit{Dyck path of semi-length $n$} is a string in $\{\U,\D\}^{2n}$ so that 
\begin{enumerate}
%\item the string starts with $\U$; 
\item the string has the same number of $\U$'s and $\D$'s (that is, $n$ of each); and
\item  the number of $\U$'s is at least the number of $\D$'s in any initial segment of the word.
\end{enumerate}
\end{Def}  

We will also often represent such a path visually using $(1,1)$ segments for $\U$ and $(-1,1)$ segments for $\D$, as in Figure \ref{fig:dyck_example}.

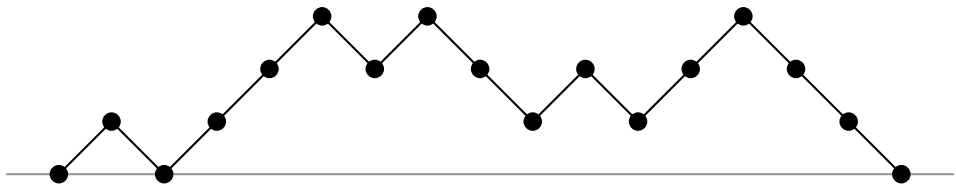
\begin{figure}[h]
\begin{center}
\begin{tikzpicture}[scale=0.7, line width=1pt]
 
\draw[color=black!40, thick] (-3,0)--(15,0); 
 
 \fill (-2,0) circle[radius=5pt];
 \draw[thick] (-2,0)--(-1,1);
 \fill (-1,1) circle[radius=5pt];
 \draw[thick] (-1,1)--(0,0);
 \fill (0,0) circle[radius=5pt];
 
 \foreach \x in {0,1,...,2}{
	\draw[thick] (\x,\x)--(\x+1,\x+1);
  	\fill (\x+1,\x+1) circle[radius=5pt];}
  	
 \draw[thick] (3,3)--(4,2);
 \fill (4,2) circle[radius=5pt];
 \draw[thick] (4,2)--(5,3);
 \fill (5,3) circle[radius=5pt];
 \draw[thick] (5,3)--(6,2);
 \fill (6,2) circle[radius=5pt];
 \draw[thick] (6,2)--(7,1);
 \fill (7,1) circle[radius=5pt];
 \draw[thick] (7,1)--(8,2);
 \fill (8,2) circle[radius=5pt];
 \draw[thick] (8,2)--(9,1);
 \fill (9,1) circle[radius=5pt];
 \draw[thick] (9,1)--(11,3);
 \fill (10,2) circle[radius=5pt];
 \fill (11,3) circle[radius=5pt];
 \draw[thick] (11,3)--(14,0);
 \fill (12,2) circle[radius=5pt];
 \fill (13,1) circle[radius=5pt];
 \fill (14,0) circle[radius=5pt];

\end{tikzpicture}
\end{center}
\caption{The visual representation of the path corresponding to $\U\D\U^3\D\U\D^2\U\D\U^2\D^3$.} \label{fig:dyck_example}
\end{figure}

\begin{Def}
A \textit{long ascent} is a maximal ascent of length at least 2. 
A \textit{singleton} is a maximal ascent of length 1. 
Let $\Dyck(n,\ell,s)$ be the Dyck paths of semi-length $n$ with $\ell$ long ascents and $s$ singletons so that no singleton appears \textit{after} the last long ascent. 
\end{Def}

Thus, the Dyck path in Figure \ref{fig:dyck_example} is an element of $\Dyck(8,2,3)$.

\subsection{Dissections and Triangulations}\label{subsec:distriback}

Throughout this section, we assume polygons with $n$ vertices have their vertices labeled $1$ through $n$ in counter-clockwise order. 

\begin{Def}
A \textit{dissection} of a polygon $P$ is a way of adding chords between non-adjacent vertices so that no two chords intersect in the interior of the polygon. 
Throughout, we let $\Dis(n,i)$ be the set of all dissections of an $n$-gon with $i$ chords. Note that $i$ in $\Dis(n,i)$ is at most $n-2$.
The elements of $\Dis(n,n-2)$ are the \emph{triangulations} of an $n$-gon. 
\end{Def}

%Given a triangulation with $n$ vertices, we number the vertices from 1 to $n$ in counter-clockwise order.  
Given a vertex $x$ in a triangulated polygon, a \textit{fan at $x$} is a maximal collection triangles all containing $x$. 
In this case, we call $x$ the \textit{origin} of the fan. 
A \textit{singular} fan is a fan with only one triangle. 
Let $e$ be a boundary edge of a fan $F$ at $x$.
% We say a fan is \textit{properly oriented relative to $e$} if, when traveling counter-clockwise along the boundary of $F$ starting at $x$, $e$ is the last edge that is reached before returning back to $x$. 

We are interested in being able to uniquely partition a triangulation into a collection of fans. This leads to the following definition. 

\begin{Def}
Let $T$ be a triangulation. A \textit{fan decomposition} is the the pair of sequences $(\mathcal{F}(T),\delta(T))$, where $\mathcal{F}(T)$ and $\delta(T)$ are defined as follows:
\begin{itemize}
\item We let $\mathcal{F}(T)$ be a sequence of fans defined recursively as follows. 
Let $F$ be the fan at the vertex with the smallest label. 
Delete this vertex and all edges incident with it in $T$ to obtain a sequence of triangulations $T_1,\dots, T_k$, arranged in counter clock-wise order so that $T_i\cap T_{i+1}$ is just vertex. 
If $T$ is just an edge, then $\mathcal{F}(T)$ is the empty sequence, and otherwise $\mathcal{F}(T)\coloneqq(F,\mathcal{F}_1,\dots, \mathcal{F}_k)$ where $\mathcal{F}_i=\mathcal{F}(T_i)$.
\item  Let $x_j$ be the label of the origin of $F_j$. We let $\delta(T)\coloneqq (d_1,\dots, d_{k-1})$, where $d_i\coloneqq x_{i+1}-x_i$ and $k$ is the number of fans in $\mathcal{F}(T)$.
\end{itemize}
\end{Def}

One can think of $d_i$ as the number of edges between the origins of $F_i$ and $F_{i+1}$ when traveling along the boundary of $T$ counter-clockwise. 

\begin{example}
Consider the triangulation $T$ in Figure \ref{fig:triangledecomp}. 
Observe that $\mathcal{F}(T)=(F_1,F_2,F_3,F_4,F_5)$ where $F_1$ is the size 1 fan at vertex 1, $F_2$ is the size 3 fan at vertex 2, $F_3$ is the size 1 fan at vertex 4, $F_4$ is the size 1 fan at vertex 5, and $F_5$ is the size 4 fan at vertex 7. Thus, $\delta(T)=(1,2,1,2)$.
Figure \ref{fig:triangledecomp} shows the five fans, distinguishing them by thick boundary edges and different shades of orange in their interior. The white vertices correspond to the origins of the fans. 

\end{example}

\begin{center}
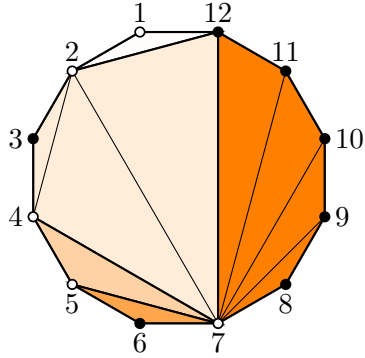
\begin{figure}[h]
\begin{tikzpicture}[scale=1]
\node[regular polygon,regular polygon sides=12,minimum size=4cm,draw,rotate=30] at (.74,.76) (e){};

\filldraw[color=orange!1] (e.corner 1)--(e.corner 2)--(e.corner 12)--(e.corner 1);
\draw[thick] (e.corner 1)--(e.corner 2)--(e.corner 12)--(e.corner 1);

\filldraw[color=orange!15] (e.corner 2)--(e.corner 12)--(e.corner 7)--(e.corner 4)--(e.corner 3)--(e.corner 2);
\draw[thick] (e.corner 2)--(e.corner 12)--(e.corner 7)--(e.corner 4)--(e.corner 3)--(e.corner 2);

\filldraw[color=orange!35] (e.corner 4)--(e.corner 5)--(e.corner 7)--(e.corner 4);
\draw[thick] (e.corner 4)--(e.corner 5)--(e.corner 7)--(e.corner 4);

\filldraw[color=orange!70] (e.corner 6)--(e.corner 5)--(e.corner 7)--(e.corner 6);
\draw[thick] (e.corner 6)--(e.corner 5)--(e.corner 7)--(e.corner 6);

\filldraw[color=orange!100] (e.corner 7)--(e.corner 8)--(e.corner 9)--(e.corner 10)--(e.corner 11)--(e.corner 12)--(e.corner 7);
\draw[thick] (e.corner 7)--(e.corner 8)--(e.corner 9)--(e.corner 10)--(e.corner 11)--(e.corner 12)--(e.corner 7);

\node[regular polygon,regular polygon sides=12,minimum size=4cm,draw,rotate=30] at (.74,.76) (e){};

\foreach \x in {1,...,12}{\node[circle,fill,inner sep=1.5pt] at (e.corner \x) {};}

\node[above] at (e.corner 1) {$1$};
\node[above] at (e.corner 2) {$2$};
\node[left] at (e.corner 3) {$3$};
\node[left] at (e.corner 4) {$4$};
\node[below] at (e.corner 5) {$5$};
\node[below] at (e.corner 6) {$6$};
\node[below] at (e.corner 7) {$7$};
\node[below] at (e.corner 8) {$8$};
\node[right] at (e.corner 9) {$9$};
\node[right] at (e.corner 10) {$10$};
\node[above] at (e.corner 11) {$11$};
\node[above] at (e.corner 12) {$12$};

{\draw (e.corner 2) -- (e.corner 12) {};}
{\draw (e.corner 2) -- (e.corner 7) {};}
{\draw (e.corner 2) -- (e.corner 4) {};}
{\draw (e.corner 4) -- (e.corner 7) {};}
{\draw (e.corner 5) -- (e.corner 7) {};}
{\draw (e.corner 6) -- (e.corner 7) {};}
{\draw (e.corner 7) -- (e.corner 10) {};}
{\draw (e.corner 7) -- (e.corner 9) {};}
{\draw (e.corner 7) -- (e.corner 11) {};}
{\draw (e.corner 7) -- (e.corner 12) {};}

\filldraw [white] (e.corner 1) circle (1.25pt);
\filldraw [white] (e.corner 2) circle (1.25pt);
\filldraw [white] (e.corner 4) circle (1.25pt);
\filldraw [white] (e.corner 5) circle (1.25pt);
\filldraw [white] (e.corner 7) circle (1.25pt);
\end{tikzpicture}
\caption{A triangulation and its partition into fans.}\label{fig:triangledecomp}
\end{figure}
\end{center}

\begin{Rmk}
Observe that a fan decomposition uniquely determines $T$. 
That is, knowing the order and size of each fan along with the distance between origins of consecutive fans uniquely determine a triangulation.
\end{Rmk}

Let $\Tri(n,t,s)$ be the triangulations $T$ of an $n$-gon so that $\mathcal{F}(T)$ has $s+t$ fans so that precisely $s$ are singular and so that the last fan is not singular.
 Thus, the triangulation in Figure \ref{fig:triangledecomp} is an element of $\Tri(12,2,3)$.

\subsection{Main Results}\label{subsec:mainback}\leavevmode \\

We now may state the main results of this paper. 
First, let us state \cite[Proposition 2.3]{GMTW20}, the result which we plan to generalize. 
We state the result by referencing the object $\Skyt(a,i,b)$ we defined above.
\begin{proposition}[\cite{GMTW20}]
\label{prop:originaldyckbijection}
The tableaux in $\Skyt(a,i,2)$ are in bijection with Dyck paths of length $2(a+2i)$ with $i+1$ peaks and no singletons.
\end{proposition}

In Section \ref{sec:combinatorial_bijections}, we will provide explicit combinatorial maps which give us the following Theorem. 

\begin{theorem}\label{thm:skyt_dyck_tri}
The following objects are in bijection: 
\begin{enumerate}
\item $\Skyt(a,i,b)$;
\item $\Dyck(a+b+2i-2,i+1,b-2)$; and
\item $\Tri(a+b+2i,i+1,b-2)$.
\end{enumerate}
\end{theorem}
\begin{proof}

%The following pairs of maps from section \ref{sec:combinatorial_bijections} are mutual inverses: 

%\suji{How about this alternate version?\GDN{accepted}}
{The maps between the three objects are defined in section \ref{sec:combinatorial_bijections}. The following pairs of maps are mutual inverses:}

\begin{itemize}
\item maps $\sd$  and $\ds$;
\item maps $\st$  and $\ts$; and
\item maps $\td$  and $\dt$.
\end{itemize}
\end{proof}

Note that this generalizes the result stated in Proposition \ref{prop:originaldyckbijection}, in addition to adding a triangulation interpretation.

With the original motivation for this paper in mind, we specialize Theorem \ref{thm:skyt_dyck_tri} to $b=2$. After incorporating the work of of \cite{S96} which provides a combinatorial bijection between $\Dis(n+2,i)$ and $\Skyt(n-i+1,i,2)$, we have the following.
\footnote{It is worth noting that this connection between $\Skyt(a,i,b)$ and dissections of polygons has resurfaced recently in the work of Kazhdan-Lusztig polynomials for Matroids \cite{EPW16}. Compare the comments in \cite[Remark 5.3]{GPY17} with the representation theoretic result \cite[Theorem 3.1]{GPY17equi} after setting $m=1$ and considering dimensions.}

%\footnote{This bijection resurfaced in some recent work of Kazhdan-Lusztig polynomials for matroids. In \cite[Remark 5.3]{GPY17} they explain that $[t^i]P_{U_{1,d}}(t)$ counts the number of ways of choosing $i$ non-intersecting chords in a $(d-i+2)$-gon. By \cite[Theorem 2]{LNRPre20}, $[t^i]P_{U_{1,d}}(t)$ also counts the cardinality of $\Skyt(2,i,d-2i+1)$ (or equivalently $\Skyt(d-2i+1,i,2)$ by symmetry as stated in Lemma \ref{lem:skew_sym}).} 

\begin{corollary}
The following objects are in bijection.
\begin{enumerate}\label{cor:disec_bij}
\item $\Dis(n+2,i)$.
\item $\Skyt(n-i+1,i,2)$.
\item $\Dyck(n+i+1,i+1,0)$. %That is, the number of Dyck paths of semilength $n+i+1$ with $i+1$ long ascents and no singleton ascents.
\item $\Tri(n+i+3,i+1,0)$. %That is, the number of triangulations of an $(n+i+3)$-gon with $i+1$ non-singleton fans.
\end{enumerate}
\end{corollary}

By specializing the maps involved in Theorem \ref{thm:skyt_dyck_tri}, we already have combinatorial bijections between the standard Young tableaux, Dyck paths, and triangulations in this theorem, though its important to note that are bijection between the standard Young tableaux and Dyck paths is precisely the proof of Proposition \ref{prop:originaldyckbijection}. 
This leaves two pairs of objects with missing combinatorial bijections. 
In section \ref{sec:enumeration} we demonstrate a bijection between the dissections and triangulations given in this corollary. 
Using our bijection between Dyck paths and Triangulations, one can extend our work in section \ref{sec:enumeration} to give a bijection between the dissections and Dyck paths in Corollary \ref{cor:disec_bij}, but we omit this interpretation from this paper.

For our final result, we recall the following Lemma in terms of Dyck paths and triangulations. 
\begin{lemma}{\cite[Lemma 5]{LNRPre20}}\label{lem:skew_sym}
Let $a$, $i$, and $b$ be nonnegative integers. Then
\[\# \Skyt(a,i,b)=\#\Skyt(b,i,a).\]
\end{lemma}
One may apply Theorem \ref{thm:skyt_dyck_tri} to this Lemma in order to get the following.

\begin{corollary}\label{cor:dyck_sym}
\leavevmode 
\begin{enumerate}
\item Let $n$, $\ell$, $s$ be nonnegative integers. Then
\[\#\Dyck(n,\ell,s)=\#\Dyck(n,\ell,n-s-2\ell).\]
\item Let $n$, $t$, $s$ be nonnegative integers. Then
\[\#\Tri(n,t,s)=\#\Tri(n,t,n-t-2\ell-2).\]
\end{enumerate}

\end{corollary}

Although these equalities are naturally obtained with Theorem \ref{thm:skyt_dyck_tri} and Lemma \ref{lem:skew_sym}, there is no known direct combinatorial bijection describing these equalities. Hence we pose the following.

\begin{problem}
Find a direct combinatorial proof of Corollary \ref{cor:dyck_sym} which does not rely on using the skew tableaux or bijections given in this paper. 
\end{problem}

\section{Combinatorial Bijections}\label{sec:combinatorial_bijections}

The following subsections describe maps going between any two of the objects given in Theorem \ref{thm:skyt_dyck_tri}. 
For convenience, we identify maps according to where the map from and to by using \texttt{S} for skew Young tableaux, \texttt{T} for Triangulations, and \texttt{D} for Dyck paths. 
For instance, map $\st$ represents the map from skew Young tableaux to traingulations, and $\td$ represents a map from triangulations to Dyck paths. 
Examples are used to alleviate any ambiguity with our maps.

Before proceeding, however, we will point out a handy reinterpretation of the tableaux in $\Skyt(a,i,b)$.
%\begin{Def}
%Let $(A_1,\dots, A_k)$ be an ordered partition of $[n]$. Let $x_j\coloneqq\min A_j$, and if $|A_j|>1$, let $y_j\coloneqq\min (A_j\setminus \{x_j\})$. We say $(A_1,\dots, A_k)$ is  \textit{nomincreasing} $x_j<x_{j+1}$ for all $j$ and $y_j<y_{j+1}$ whenever both $y_j$ and $y_{j+1}$ exist. 
%\end{Def}
%For instance, if $A_1=\{1,3,5\}$, $A_2=\{2\}$, $A_3=\{4,6\}$, and $A_4=\{7,8\}$, then $(A_1,A_2,A_3,A_4)$ is a nomincreasing partition of $[8]$. 
%There is a natural bijection between nomincreasing partitions of $[a+b+2i-2]$ with $b+i-1$ parts, $b-2$ of which have cardinality 1, and a tableaux of shape $\Skyt(a,i,b)$ which we will readily use through out this paper which will quickly help us form bijections with these tableaux and other objects. 
Let $\lambda\in \Skyt(a,i,b)$. Let $X=\{x_1,x_2,\dots, x_{i+b-1}\}$ be the set of values in the top $b-1$ rows so that $x_1<x_2<\cdots < x_{i+b-1}$. 
If $x_j$ is in row $b-1$, define $y_j$ to be the entry in the tableau directly below $x_j$.  Then for $1\leq j<i+b-1$ let 
\[A_{j}\coloneqq\begin{cases} \{x_j\} & \text{if $x_j$ is in the first $b-2$ rows;}\\
\{x_j\}\cup\left([y_j,y_{k}-1]\setminus X\right)& \text{if $x_j$ is in row $b-1$ and $y_k$ is to the right of $y_j$,}\end{cases}\]
where $[y_j,y_k-1]=\{y_j,y_j+1,y_j+2,\dots, y_k-1\}$. Let $A_{i+b-1}\coloneqq\{x_j\}\cup \big([x_j+1,a+b+2i-2]\setminus X\big)$. Note that $x_j$ is always the minimum of $A_j$. When $|A_j|>1$, note the elements of $A_j$ are precisely the entries in row $b-1$ and $b$ in column $j$ along with all entries of column 1 which are between $y_j$ and $y_k$. See Figure \ref{fig:pushidea}.

\begin{figure}[h]
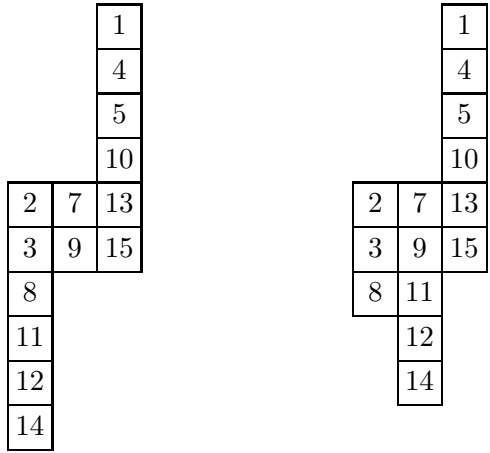

\begin{ytableau}
       \none & \none & 1\\
       \none & \none & 4 \\
       \none & \none & 5 \\
       \none & \none & 10 \\
       2 & 7 & 13 \\
       3 & 9 & 15  \\
       8  \\
       11  \\
       12 \\
       14 
\end{ytableau}
\hspace{1in}
\begin{ytableau}
       \none & \none & 1\\
       \none & \none & 4 \\
       \none & \none & 5 \\
       \none & \none & 10 \\
       2 & 7 & 13 \\
       3 & 9 & 15  \\
       8  & 11 \\
        \none & 12  \\        
        \none& 14
\end{ytableau}
\caption{We ``push" entries below row $b$ as far right as possible while maintaining the property that columns increase from top to bottom. We have $A_1=\{1\}$, $A_2=\{2,3,8\}$, $A_3=\{4\}$, $A_4=\{5\}$, $A_5=\{7,9,11,12,14\}$, $A_6=\{10\}$, $A_7=\{13,15\}$.}\label{fig:pushidea}
\end{figure}

The sequence $(A_1,\dots, A_{i+b-1})$ has enough information to reconstruct $\lambda$. 
%Observe by construction have $\min A_j=x_j<x_{j+1}=\min A_{j+1}$. Also, if $A_j$ and $A_{j+1}$ have cardinality larger than $1$, then \[\min (A_j\setminus \{x_j\})=y_j<y_{j+1}\min (A_{j+1}\setminus \{x_{j+1}\}),\]
%which is due to the fact that the rows of $\lambda$ increase from left-to-right. Thus $(A_1,\dots, A_{i+b-1})$ is a nomincreasing partition of $[a+b+2i-2]$. 
%Conversely, given a nomincreasing partition $(A_1,\dots, A_{i+b-1})$ of $[a+b+2i-2]$ with precisely $b-2$ parts of size $1$, one gets a tableaux in $\Skyt(a,i,b)$. Let $\lambda$ be a skew diagram of the desired shape. 
Starting with $j=1$, do the following. 

\begin{enumerate}
\item If $|A_j|=1$, let $x\in A_j$. Then place $x$ in the highest possible position in the last column. 

\item If $|A_j|>1$, then let $x_j=\min A_j$ and $y_j=\min( A_j\setminus \{x_j\})$. Place $x_j$ in row $b-1$ column $j$ and place $y_j$ in row $b$ column $j$. Place all remaining entries from $A_j$---in increasing order---at the top most available position(s) in the first column.

\item Increase the value of $j$ by 1. If $j<i+b-1$, repeat these steps. Otherwise, $\lambda$ is filled and we are done.
\end{enumerate}

%Note that we will always have $x_j<x_{j+1}$ and $y_j<y_{j+1}$, and thus rows in $\lambda$ increase from left-to-right. Also by construction columns increase from top-to-bottom. Since there are precisely $b-2$ of the $A_j$ that have cardinality $1$, we have that $\lambda\in \Skyt(a,i,b)$. 

%Note the processes we have described are mutual inverses, thus we have a bijection between $\Skyt(a,i,b)$ and nomincreasing partitions of $[a+b+2i-2]$ with $i+b-1$ parts, where $b-2$ have cardinality $b-2$. We will readily move between these two interpretations in the following proofs.

Let $x_j$ and $y_j$ are defined in step (2) of the preceding procedure. Pick an integer $j'$ minimally so that $j<j'$ and $|A_{j'}|>1$. Note that $(A_1,\dots, A_{i+b-1})$ is an ordered partition of $[a+b+2i-2]$ so that $x_j<x_{j+1}$ and $y_j<y_{j'}$, whenever $y_j$ and $y_{j'}$ exist. These conditions guarantee that the rows of $\lambda$ increase left-to-right. Such a sequence is called a \textit{nomincreasing partition}, as defined in \cite{GMTW20}. To this end, we define the following.
\begin{Def}
Let $\lambda\in \Skyt(a,i,b)$ and define $A_1,\dots, A_{i+b-1}$ for $\lambda$ as above. We define $\Nom(\lambda)$ to be the nomincreasing partition
\[\Nom(\lambda)\coloneqq (A_1,A_2,\dots, A_{i+b-1}).\]
\end{Def}

\begin{Rmk}
\editone{One thing that will be useful to note for future reference is that for $\Nom(\lambda)=(A_1,\dots, A_{i+b-1})$, we always have $|A_{i+b-1}|>1$. This is preceisely because with the aforementioned choice of $x_1<\cdots <x_{i+b-1}$, we always have that $x_{i+b-1}$ is the entry in row $b-1$ column $i+1$ of $\lambda$. In particular, this means given a nomincreasing sequence $(A_1,\dots, A_{i+b-1})$, even if $i+1$ of the $A_j$ satisfy $|A_j|>1$ and the remaining satisfy $|A_j|=1$, there does not necessarily exists $\lambda\in \Skyt(a,i,b)$ so that $\Nom(\lambda)=(A_1,\dots, A_{i+b-1})$. We need to additionally have $|A_{i+b-1}|>1$. Given this, though, such a $\lambda$ must exist.}
\end{Rmk}

\subsection{Map $\sd$}\label{map_sd}

In this section, we define a map $\sd$ from $\Skyt(a,i,b)$ to $\Dyck(a+b+2i-2, i+1, b-2)$. 
For simplicity, let $n\coloneqq a+b+2i-2$.
Note that given $\lambda \in \Skyt(a,i,b)$, $n$ is the number of entries in $\lambda$ and $t$ is the number of entries in the first $b-1$ rows of $\lambda$.

\begin{Def}
Let $\lambda\in \Skyt(a,i,b)$. We define $\sd(\lambda)$, a certain lattice path, as follows.

Let $\Nom(\lambda)=(A_1,\dots, A_{i+b-1})$. Let $x_j$ denote the minimum of $A_j$.
Let $a_j \coloneqq \#A_j$. 
Then let $\sd(\lambda)$ be the lattice path given by following string in $\{\U,\D\}^{2n}$: 
\begin{align}
\U^{a_1}\D^{x_2-x_1}\U^{a_2}\D^{x_3-x_2}\cdots \U^{a_{t-1}}\D^{x_{t}-x_{t-1}}\U^{a_{t}}\D^{n-x_{t}+x_1}.\label{eq:dyckstring}
\end{align}

\end{Def}

%We claim the string given in \eqref{eq:dyckstring} is in fact a Dyck path. 

\begin{lemma}
Given $\lambda$ be a tableau in $\Skyt(a,i,b)$, $\sd(\lambda)$ is a Dyck path. In particular, the lattice path $\sd(\lambda)$ is an element of $\Dyck(n, i+1, b-2)$. 

\begin{proof}
Recall that given $\lambda \in \Skyt(a,i,b)$, we can define have $\Nom(\lambda)=(A_1,\dots, A_{i+b-1})$ where $|A_{i+b-1}|>1$. Recall that $x_j=\min A_j$ is an entry in the top $b-1$ rows of $\lambda$. 
 
In the string given in \eqref{eq:dyckstring}, $\U$ corresponds to an up step and $\D$ correspond to a down step. 
Note that for any $k$, we have $[x_1,x_k]\subseteq A_1\cup \cdots \cup A_k$. 
Otherwise, there exists a $w\in[x_1,x_k]$ so that $w\in A_j$ for some $j>k$. Then we have $w\geq x_j>x_k\geq w$, a contradiction.

Thus, for $k<t$
\[\sum_{j=1}^k (x_{j+1}-x_j)=x_{k+1}-x_1\leq \sum_{j=1}^k a_j.\]

Also,
\[\sum_{j=1}^t (x_{j+1}-x_j)=x_{t}-x_1 + n-x_t+x_1 = n.\]

Moreover, there are precisely $b-2$ of the $a_j$ so that $a_j=1$, and there are $i$ of the $a_j$ so that $a_j>1$. \editone{Also, observe that $a_{i+b-1}>1$, so the last ascent in $\sd(\lambda)$ is not a singleton. }
Consequently, the constructed Dyck path is an element of $\Dyck(n,i,b-2)$.

\end{proof}
\end{lemma}

\begin{example}
For the skew Young tableau $\lambda$ in $\Skyt(7,2,6)$ below, $x_1=1, x_2=2, x_3=4,x_4=5,x_5=7 , x_6=10,$ and $x_7=12$. 
As $x_2,$ $x_5$ and $x_7$ are the entries in row 5, $y_i$ is defined for $i=2,5,7$. We have $y_2=3$, $y_5=11$, and $y_7=13$. 
Thus, $A_1 =\{1\}, A_2 =\{2,3,6,8,9\}, A_3 =\{4\}, A_4 =\{5\}, A_5 =\{7,11\}, A_6 =\{10\}$, and $A_7 =\{12,13,14,15\}$.
\begin{center}

\begin{minipage}{.49\textwidth}

\begin{center}
\begin{tikzpicture}[scale=0.5, line width=1pt]
 \draw (0,0) grid (1,-7);
 \draw (1,0) grid (2,-2);
 \draw (2,0) grid (3,-2);
 \draw (2,4) grid (3,-2);

\draw (.5,-.5) node {$2$}; 
\draw (.5,-1.5) node {$3$}; 
\draw (.5,-2.5) node {$6$}; 
\draw (.5,-4.5) node {$9$}; 
\draw (.5,-3.5) node {$8$}; 
\draw (1.5,-.5) node {$7$};
\draw (1.5,-1.5) node {$11$}; 
\draw (.5,-5.5) node {$14$}; 
\draw (.5,-6.5) node {$15$}; 
\draw (2.5,-.5) node {$12$}; 
\draw (2.5,-1.5) node {$13$}; 

\draw (2.5,3.5) node {$1$}; 
\draw (2.5,2.5) node {$4$}; 
\draw (2.5,1.5) node {$5$}; 
\draw (2.5,.5) node {$10$};
 
\end{tikzpicture}
\end{center}
\end{minipage}
\end{center}
Thus $\sd(\lambda)$ is the Dyck path given by \[\U\D\U^3\D^2{{}\U\D\U\D}\D\U^4\D^3{{}\U\D}\D\U^4\D^4.\]

\end{example}

%Note that \[\{m_1,m_1+1,\dots, m_{j}-1\}\subseteq S_1\cup \cdots \cup S_{j-1}\cup N_2\cup \cdots \cup N_j,\] so in particular the above path is indeed a Dyck path. Observe each of the $\U\D$ in the $(\U\D)^{n_j}$ correspond to one of the entries in the top $b-2$ positions of the right column in our original tableau. Also note that the number of long ascents in precisely $i+1$ and the number of singletons in $b-2$. Let $x$ be such an entry. If needed, shift the copy of $\U\D$ corresponding to $x$ to the left to guarantee its $\D$ is the $x$th to appear when counting the $\D$s from left to right in the above string. Despite shifting, note that this Dyck path will not have a singleton after the last long ascent as $s_{i+1}\geq 2$. Thus, the resulting Dyck path is an element of $\Dyck(a+b+2i-2,i+1,b-2)$.

%For example, before shifting, the above tableau gives the string 
%\[\U\D\U^3\D^3{\color{red}\U\D\U\D}\U^4\D^4{\color{red}\U\D}\U^4\D^4,\]
%where singletons are colored red.

%After shifting the singletons, we get

\subsection{Map $\ds$} \label{map_ds}
\editone{This subsection gives a map $\ds$ from $\Dyck(n,\ell,s)$ to $\Skyt(n-s-2\ell-2,\ell-1,s+2)$, which is the inverse of $\sd.$} The reader can verify that they are indeed inverses.
\begin{Def}
\editone{Let $P$ be a Dyck path in $\Dyck(n,\ell,s)$. We define $\ds(P)$, a skew tableau, as follows.}  

Given $P$, label the down-steps, left to right, in increasing order, from 1 to $n$. 
Next, use the label on the down-step at each peak as the label for the up-step at the same peak. 
\editone{Going through the ascents from left to right, greedily label the unlabeled up-steps from top to bottom using numbers from $[a+2i+b-2]$ not already appearing on any up-step.}

\editone{Let $A_j$ be the labels appearing on the $j$th ascent. Now construct $\ds(P)$ so that $\Nom\big(\ds(P)\big)=(A_1,\dots, A_{\ell+s})$.}
%Now construct a tableau diagram $\lambda$ of shape $(2^i,1^{n-s-2l})$.  Now insert the smallest entry of $S_j$ in row 1 column $j$ in $\lambda$. Place the other entry of $S_j$ in row 2 column $j$ of $\lambda$. Place the entries of $S$ in increasing order---from top to bottom---in the unfilled cells in the first column. Then attach a diagram of shape $(1^{\#N})$ to the top right cell of $\lambda$ and fill it with the entries of $N$ in increasing order, again from top to bottom. 

\end{Def}

\begin{lemma}
Given a Dyck path $P$ in $\Dyck(n,\ell,s)$, the tableau $\ds(P)$ is a skew Young tableau in $\Skyt(n-s-2\ell-2,\ell-1,\editone{s+2})$. 
\begin{proof}
\editone{There are $\ell+s$ ascents in $P$. Thus, precisely $\ell$ of the $A_j$ satisfy $|A_j|>1$, and precisely $s$ of the $A_j$ satisfy $|A_j|=1$.  
Next, notice that $(A_1,\dots, A_{\ell+s})$ is a nomincreasing sequence due to the greedy labeling of up steps of $P$. Also note that $|A_{s+\ell}|>1$ since the last ascent in  $P$ is not a singleton.
%We also have that the first column must have $n-s-2\ell-2$ entries. 
%Finally, note that the construction of the $A_j$ from section \ref{map_sd} guarantees that the entries increase left-to-right and top-to-bottom. 
Since there are $n$ up steps in $P$, we have $\ds(P)\in \Skyt(n-s-2\ell-2,\ell-1,s+2)$. }

\end{proof}
\end{lemma}

\begin{example}
Given a Dyck path in $\Dyck(15,3,4)$, we first label the down-steps, left to right, in increasing order. 
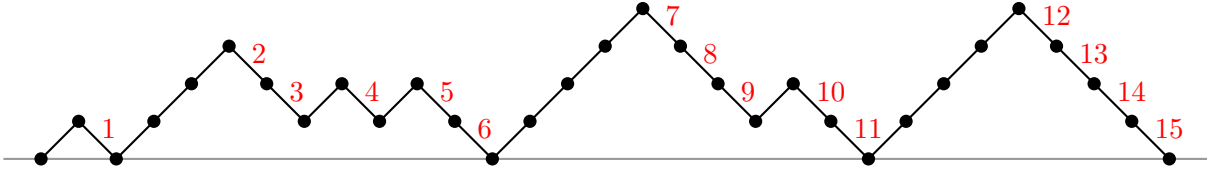
\begin{center}
\begin{tikzpicture}[scale=0.5, line width=1pt]
 
\draw[color=black!40, thick] (-3,0)--(29,0); 
 
 \fill (-2,0) circle[radius=5pt];
 \draw[thick] (-2,0)--(-1,1);
 \fill (-1,1) circle[radius=5pt];
 \draw[thick] (-1,1)--(0,0);
 \fill (0,0) circle[radius=5pt];
 
 \foreach \x in {0,1,...,2}{
	\draw[thick] (\x,\x)--(\x+1,\x+1);
  	\fill (\x+1,\x+1) circle[radius=5pt];}
  	
 \draw[thick] (3,3)--(4,2);
 \fill (4,2) circle[radius=5pt];
 \draw[thick] (4,2)--(5,1);
 \fill (5,1) circle[radius=5pt];
 \draw[thick] (5,1)--(6,2);
 \fill (6,2) circle[radius=5pt];
 \draw[thick] (6,2)--(7,1);
 \fill (7,1) circle[radius=5pt];
 \draw[thick] (7,1)--(8,2);
 \fill (8,2) circle[radius=5pt];
 \draw[thick] (8,2)--(9,1);
 \fill (9,1) circle[radius=5pt];
 \draw[thick] (9,1)--(10,0);
 \fill (10,0) circle[radius=5pt];
 \draw[thick] (10,0)--(14,4);
 \fill (11,1) circle[radius=5pt];
 \fill (12,2) circle[radius=5pt];
 \fill (13,3) circle[radius=5pt];
 \fill (14,4) circle[radius=5pt];
 \draw[thick] (14,4)--(17,1);
 \fill (15,3) circle[radius=5pt];
 \fill (16,2) circle[radius=5pt];
 \fill (17,1) circle[radius=5pt];
 
 \draw[thick] (17,1)--(18,2);
 \fill (18,2) circle[radius=5pt];
 \draw[thick] (18,2)--(20,0);
 \fill (19,1) circle[radius=5pt];
 \fill (20,0) circle[radius=5pt];
 \draw[thick] (20,0)--(24,4);
 \fill (21,1) circle[radius=5pt];
 \fill (22,2) circle[radius=5pt];
 \fill (23,3) circle[radius=5pt];
 \fill (24,4) circle[radius=5pt];
 \draw[thick] (24,4)--(28,0);
 \fill (25,3) circle[radius=5pt];
 \fill (26,2) circle[radius=5pt];
 \fill (27,1) circle[radius=5pt];
 \fill (28,0) circle[radius=5pt];

\draw (-.2,.8) node[color=red] {$1$}; 
\draw (3.8,2.8) node[color=red] {$2$}; 
\draw (4.8,1.8) node[color=red] {$3$}; 
\draw (6.8,1.8) node[color=red] {$4$}; 
\draw (8.8,1.8) node[color=red] {$5$}; 
\draw (9.8,.8) node[color=red] {$6$}; 
\draw (14.8,3.8) node[color=red] {$7$};
\draw (15.8,2.8) node[color=red] {$8$};
\draw (16.8,1.8) node[color=red] {$9$};
\draw (19,1.8) node[color=red] {$10$}; 
\draw (20,0.8) node[color=red] {$11$}; 
\draw (25.,3.8) node[color=red] {$12$}; 
\draw (26.,2.8) node[color=red] {$13$}; 
\draw (27.,1.8) node[color=red] {$14$}; 
\draw (28.,.8) node[color=red] {$15$}; 
 
\end{tikzpicture}
\end{center}

Then label the upstep of each peak: 

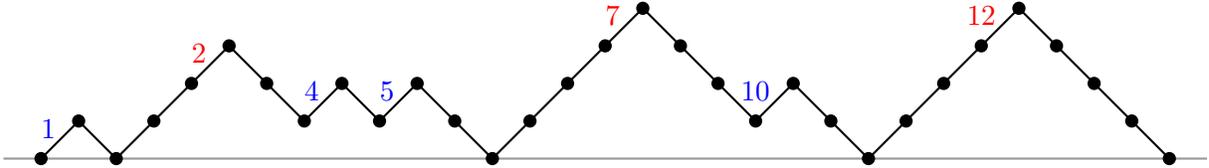
\begin{center}
\begin{tikzpicture}[scale=0.5, line width=1pt]
 
\draw[color=black!40, thick] (-3,0)--(29,0); 
 
 \fill (-2,0) circle[radius=5pt];
 \draw[thick] (-2,0)--(-1,1);
 \fill (-1,1) circle[radius=5pt];
 \draw[thick] (-1,1)--(0,0);
 \fill (0,0) circle[radius=5pt];
 
 \foreach \x in {0,1,...,2}{
	\draw[thick] (\x,\x)--(\x+1,\x+1);
  	\fill (\x+1,\x+1) circle[radius=5pt];}
  	
 \draw[thick] (3,3)--(4,2);
 \fill (4,2) circle[radius=5pt];
 \draw[thick] (4,2)--(5,1);
 \fill (5,1) circle[radius=5pt];
 \draw[thick] (5,1)--(6,2);
 \fill (6,2) circle[radius=5pt];
 \draw[thick] (6,2)--(7,1);
 \fill (7,1) circle[radius=5pt];
 \draw[thick] (7,1)--(8,2);
 \fill (8,2) circle[radius=5pt];
 \draw[thick] (8,2)--(9,1);
 \fill (9,1) circle[radius=5pt];
 \draw[thick] (9,1)--(10,0);
 \fill (10,0) circle[radius=5pt];
 \draw[thick] (10,0)--(14,4);
 \fill (11,1) circle[radius=5pt];
 \fill (12,2) circle[radius=5pt];
 \fill (13,3) circle[radius=5pt];
 \fill (14,4) circle[radius=5pt];
 \draw[thick] (14,4)--(17,1);
 \fill (15,3) circle[radius=5pt];
 \fill (16,2) circle[radius=5pt];
 \fill (17,1) circle[radius=5pt];
 
 \draw[thick] (17,1)--(18,2);
 \fill (18,2) circle[radius=5pt];
 \draw[thick] (18,2)--(20,0);
 \fill (19,1) circle[radius=5pt];
 \fill (20,0) circle[radius=5pt];
 \draw[thick] (20,0)--(24,4);
 \fill (21,1) circle[radius=5pt];
 \fill (22,2) circle[radius=5pt];
 \fill (23,3) circle[radius=5pt];
 \fill (24,4) circle[radius=5pt];
 \draw[thick] (24,4)--(28,0);
 \fill (25,3) circle[radius=5pt];
 \fill (26,2) circle[radius=5pt];
 \fill (27,1) circle[radius=5pt];
 \fill (28,0) circle[radius=5pt];

\draw (-1.8,.8) node[color=blue] {$1$}; 
\draw (2.2,2.8) node[color=red] {$2$}; 
\draw (5.2,1.8) node[color=blue] {$4$}; 
\draw (7.2,1.8) node[color=blue] {$5$}; 
\draw (13.2,3.8) node[color=red] {$7$};
\draw (17.,1.8) node[color=blue] {$10$}; 
\draw (23.,3.8) node[color=red] {$12$}; 
 
\end{tikzpicture}
\end{center}

\editone{Now greedily label remaining up-steps. }

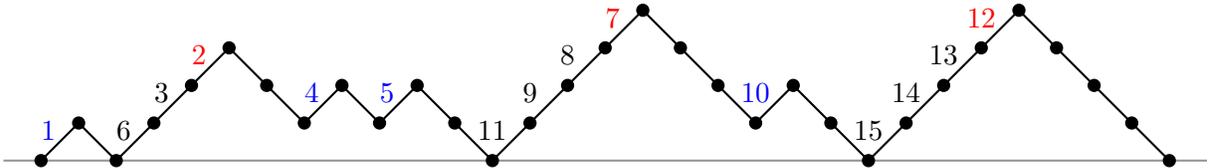
\begin{center}
\begin{tikzpicture}[scale=0.5, line width=1pt]
 
\draw[color=black!40, thick] (-3,0)--(29,0); 
 
 \fill (-2,0) circle[radius=5pt];
 \draw[thick] (-2,0)--(-1,1);
 \fill (-1,1) circle[radius=5pt];
 \draw[thick] (-1,1)--(0,0);
 \fill (0,0) circle[radius=5pt];
 
 \foreach \x in {0,1,...,2}{
	\draw[thick] (\x,\x)--(\x+1,\x+1);
  	\fill (\x+1,\x+1) circle[radius=5pt];}
  	
 \draw[thick] (3,3)--(4,2);
 \fill (4,2) circle[radius=5pt];
 \draw[thick] (4,2)--(5,1);
 \fill (5,1) circle[radius=5pt];
 \draw[thick] (5,1)--(6,2);
 \fill (6,2) circle[radius=5pt];
 \draw[thick] (6,2)--(7,1);
 \fill (7,1) circle[radius=5pt];
 \draw[thick] (7,1)--(8,2);
 \fill (8,2) circle[radius=5pt];
 \draw[thick] (8,2)--(9,1);
 \fill (9,1) circle[radius=5pt];
 \draw[thick] (9,1)--(10,0);
 \fill (10,0) circle[radius=5pt];
 \draw[thick] (10,0)--(14,4);
 \fill (11,1) circle[radius=5pt];
 \fill (12,2) circle[radius=5pt];
 \fill (13,3) circle[radius=5pt];
 \fill (14,4) circle[radius=5pt];
 \draw[thick] (14,4)--(17,1);
 \fill (15,3) circle[radius=5pt];
 \fill (16,2) circle[radius=5pt];
 \fill (17,1) circle[radius=5pt];
 
 \draw[thick] (17,1)--(18,2);
 \fill (18,2) circle[radius=5pt];
 \draw[thick] (18,2)--(20,0);
 \fill (19,1) circle[radius=5pt];
 \fill (20,0) circle[radius=5pt];
 \draw[thick] (20,0)--(24,4);
 \fill (21,1) circle[radius=5pt];
 \fill (22,2) circle[radius=5pt];
 \fill (23,3) circle[radius=5pt];
 \fill (24,4) circle[radius=5pt];
 \draw[thick] (24,4)--(28,0);
 \fill (25,3) circle[radius=5pt];
 \fill (26,2) circle[radius=5pt];
 \fill (27,1) circle[radius=5pt];
 \fill (28,0) circle[radius=5pt];

\draw (-1.8,.8) node[color=blue] {$1$}; 
\draw (2.2,2.8) node[color=red] {$2$}; 
\draw (1.2,1.8) node {$3$}; 
\draw (0.2,0.8) node {$6$}; 
\draw (5.2,1.8) node[color=blue] {$4$}; 
\draw (7.2,1.8) node[color=blue] {$5$}; 
\draw (13.2,3.8) node[color=red] {$7$};
\draw (12.,2.8) node {$8$}; 
\draw (11.,1.8) node {$9$}; 
\draw (10.,.8) node {$11$}; 
\draw (17.,1.8) node[color=blue] {$10$}; 
\draw (23.,3.8) node[color=red] {$12$}; 
\draw (22.,2.8) node {$13$}; 
\draw (21.,1.8) node {$14$}; 
\draw (20.,.8) node {$15$}; 
 
\end{tikzpicture}
\end{center}
Thus, we have $A_1=\{1\}$, $A_2=\{2,3,6\}$, $A_3=\{4\}$, $A_4=\{5\}$, $A_5=\{7,8,9,11\}$, $A_6=\{10\}$, and $A_7=\{12,13,14,15\}$. The map $\dt$ gives the tableau in Figure \ref{fig:dyck_to_skew}.

\begin{center}

\begin{figure}[h]
%\begin{minipage}{.3\textwidth}
%
%\begin{center}
%\begin{tikzpicture}[scale=0.5, line width=1pt]
% \draw (0,0) grid (1,-7);
% \draw (1,0) grid (2,-2);
% \draw (2,0) grid (3,-2);
%
%\draw (.5,-.5) node[color=red] {$2$}; 
%\draw (.5,-1.5) node {$3$}; 
%\draw (.5,-2.5) node {$6$}; 
%\draw (.5,-4.5) node {$11$}; 
%\draw (.5,-3.5) node {$9$}; 
%\draw (1.5,-.5) node[color=red] {$7$};
%\draw (1.5,-1.5) node {$8$}; 
%\draw (.5,-5.5) node {$14$}; 
%\draw (.5,-6.5) node {$15$}; 
%\draw (2.5,-.5) node[color=red] {$12$}; 
%\draw (2.5,-1.5) node {$13$}; 
% 
%\end{tikzpicture}
%\end{center}
%\end{minipage}
%\begin{minipage}{.3\textwidth}

\begin{center}
\begin{tikzpicture}[scale=0.45, line width=1pt]
 \draw (0,0) grid (1,-7);
 \draw (1,0) grid (2,-2);
 \draw (2,0) grid (3,-2);
 \draw (2,4) grid (3,-2);

\draw (.5,-.5) node[color=red] {$2$}; 
\draw (.5,-1.5) node {$3$}; 
\draw (.5,-2.5) node {$6$}; 
\draw (.5,-4.5) node {$11$}; 
\draw (.5,-3.5) node {$9$}; 
\draw (1.5,-.5) node[color=red] {$7$};
\draw (1.5,-1.5) node {$8$}; 
\draw (.5,-5.5) node {$14$}; 
\draw (.5,-6.5) node {$15$}; 
\draw (2.5,-.5) node[color=red] {$12$}; 
\draw (2.5,-1.5) node {$13$}; 

\draw (2.5,3.5) node[color=blue] {$1$}; 
\draw (2.5,2.5) node[color=blue] {$4$}; 
\draw (2.5,1.5) node[color=blue] {$5$}; 
\draw (2.5,.5) node[color=blue] {$10$};
 
\end{tikzpicture}
\end{center}
%\end{minipage}
\caption{The construction of the skew tableaux in the final steps of the $\ds$ map.}
\label{fig:dyck_to_skew}
\end{figure}

\end{center}
\end{example}

\subsection{Map $\dt$}\label{map_dt}
\editone{The inspiration for the following map $\dt$, a map from $\Dyck(n,\ell,s)$ to $\Tri(n+2,\ell,s)$, comes from \cite[Proposition 6.2.1]{StanleyBook}. }

\begin{Def}
\editone{Let $P$ be a Dyck path in $\Dyck(n,\ell,s)$. We define $\dt(P)$ a certain type of triangulation, as follows. }

A Dyck path in $\Dyck(n,\ell,s)$ has the form
\[\U^{u_1}\D^{d_1}\U^{u_2}\D^{d_2}\cdots \U^{u_{s+\ell}}\D^{d_{s+\ell}},\]
where $s$ is the number of singletons, $\ell$ is the number of long ascents, and the $u_i$ and $d_i$ are positive integers. Recall that the triangulation is determined by its fan decomposition. Let $F_j$ be a fan with $u_j$ triangles. Then $\dt(P)$ is given by the fan decomposition \editone{$\big( (F_1,\dots, F_{s+\ell}), (d_1,\dots, d_{s+\ell})\big)$. }
\end{Def}
\begin{lemma}
If $P$ is a Dyck path in $\Dyck(n,\ell,s)$, then $\dt(P)$ is a triangulation in $\Tri(n+2,\ell,s)$.
\begin{proof}

Given $P\in \Dyck(n,\ell,s)$, note that that the number of triangles in $\dt(P)$ is given by sum of sizes of $F_j$:
\[\sum_{j=1}^{s+\ell}u_j=n\] 

As $\dt(P)$ has $n$ triangles, the boundary must have $n+2$ edges. Also note that since there is no singleton after the last long ascent in our Dyck path, the last fan $\cF(T)$ will not be a singleton fan. \editone{Since $\ell$ of the $u_j$ satisfy $u_j>1$, our triangulation has $\ell$ non-singular fans. Similarly, since $s$ of the $u_j$ satisfy $u_j=1$, our triangulation has $s$ singular fans. Also, note that $u_{s+\ell}>1$, so $F_{s+\ell}$ has more than one triangle.} Thus, we have constructed a triangulation in $\Tri(n+2,\ell,s)$, where the vertices are labeled as follows: label the origin of $F_1$ as 1, and then label the remaining vertices from 2 to $n$ in clock-wise order by starting at 1 and traveling along the boundary of the triangulation.

\end{proof}
\end{lemma}

\begin{example}
Consider the path below. 
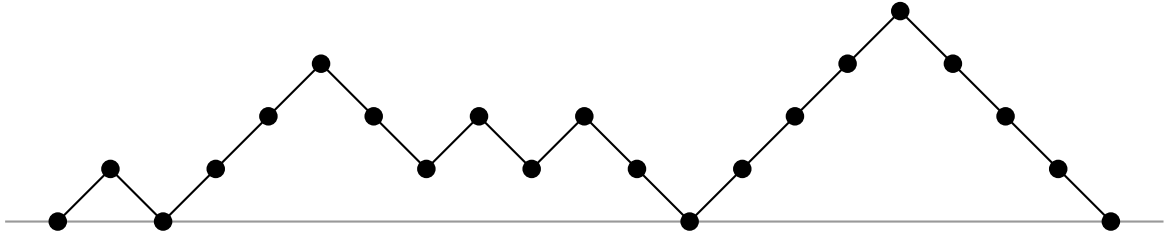
\begin{center}
\begin{tikzpicture}[scale=0.7, line width=1pt]
 
\draw[color=black!40, thick] (-3,0)--(19,0); 
 
 \fill (-2,0) circle[radius=5pt];
 \draw[thick] (-2,0)--(-1,1);
 \fill (-1,1) circle[radius=5pt];
 \draw[thick] (-1,1)--(0,0);
 \fill (0,0) circle[radius=5pt];
 
 \foreach \x in {0,1,...,2}{
	\draw[thick] (\x,\x)--(\x+1,\x+1);
  	\fill (\x+1,\x+1) circle[radius=5pt];}
  	
 \draw[thick] (3,3)--(4,2);
 \fill (4,2) circle[radius=5pt];
 \draw[thick] (4,2)--(5,1);
 \fill (5,1) circle[radius=5pt];
 \draw[thick] (5,1)--(6,2);
 \fill (6,2) circle[radius=5pt];
 \draw[thick] (6,2)--(7,1);
 \fill (7,1) circle[radius=5pt];
 \draw[thick] (7,1)--(8,2);
 \fill (8,2) circle[radius=5pt];
 \draw[thick] (8,2)--(9,1);
 \fill (9,1) circle[radius=5pt];
 \draw[thick] (9,1)--(10,0);
 \fill (10,0) circle[radius=5pt];
 \draw[thick] (10,0)--(14,4);
 \fill (11,1) circle[radius=5pt];
 \fill (12,2) circle[radius=5pt];
 \fill (13,3) circle[radius=5pt];
 \fill (14,4) circle[radius=5pt];
 \draw[thick] (14,4)--(18,0);
 \fill (15,3) circle[radius=5pt];
 \fill (16,2) circle[radius=5pt];
 \fill (17,1) circle[radius=5pt];
 \fill (18,0) circle[radius=5pt];

\end{tikzpicture}
\end{center}

Associated to this are five fans, given below. \editone{We shade these fans differently so we may more easily keep track of them throughout.}
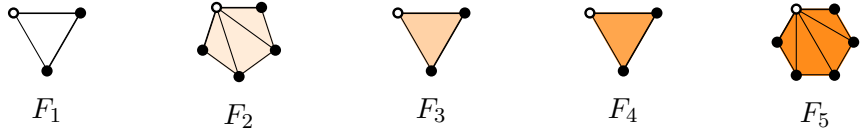
\begin{center}
\begin{minipage}{.15\textwidth}
\begin{center}
\begin{tikzpicture}
\node[regular polygon,regular polygon sides=3,minimum size=1cm,draw,rotate=180] (a){};
\foreach \x in {1,...,3}{\node[circle,fill,inner sep=1.5pt] at (a.corner \x) {};}
\foreach \x in {1,...,3}
{\draw (a.corner 2) -- (a.corner \x) {};}
\filldraw [white] (a.corner 3) circle (1.pt);
\end{tikzpicture}

$F_1$
\end{center}
\end{minipage}
\begin{minipage}{.15\textwidth}
\begin{center}
\begin{tikzpicture}
\node[regular polygon,regular polygon sides=5,minimum size=1cm,draw,rotate=-36] (a){};
\filldraw[color=orange!15] (a.corner 1)--(a.corner 2)--(a.corner 3)--(a.corner 4)--(a.corner 5)--(a.corner 1);
\node[regular polygon,regular polygon sides=5,minimum size=1cm,draw,rotate=-36] (a){};
\foreach \x in {1,...,5}{\node[circle,fill,inner sep=1.5pt] at (a.corner \x) {};}
\foreach \x in {1,...,5}
{\draw (a.corner 2) -- (a.corner \x) {};}
\filldraw [white] (a.corner 2) circle (1.pt);
\end{tikzpicture}

$F_2$
\end{center}
\end{minipage}
\begin{minipage}{.15\textwidth}
\begin{center}
\begin{tikzpicture}
\node[regular polygon,regular polygon sides=3,minimum size=1cm,draw,rotate=180] (a){};
\filldraw[color=orange!35] (a.corner 1)--(a.corner 2)--(a.corner 3)--(a.corner 1);
\node[regular polygon,regular polygon sides=3,minimum size=1cm,draw,rotate=180] (a){};
\foreach \x in {1,...,3}{\node[circle,fill,inner sep=1.5pt] at (a.corner \x) {};}
\foreach \x in {1,...,3}
{\draw (a.corner 2) -- (a.corner \x) {};}
\filldraw [white] (a.corner 3) circle (1.pt);
\end{tikzpicture}

$F_3$
\end{center}
\end{minipage}
\begin{minipage}{.15\textwidth}
\begin{center}
\begin{tikzpicture}
\node[regular polygon,regular polygon sides=3,minimum size=1cm,draw,rotate=180] (a){};
\filldraw[color=orange!70] (a.corner 1)--(a.corner 2)--(a.corner 3)--(a.corner 1);
\node[regular polygon,regular polygon sides=3,minimum size=1cm,draw,rotate=180] (a){};
\foreach \x in {1,...,3}{\node[circle,fill,inner sep=1.5pt] at (a.corner \x) {};}
\foreach \x in {1,...,3}
{\draw (a.corner 2) -- (a.corner \x) {};}
\filldraw [white] (a.corner 3) circle (1.pt);
\end{tikzpicture}

$F_4$
\end{center}
\end{minipage}
\begin{minipage}{.15\textwidth}
\begin{center}
\begin{tikzpicture}
\node[regular polygon,regular polygon sides=6,minimum size=1cm,draw] (a){};

\filldraw[color=orange!90] (a.corner 1)--(a.corner 2)--(a.corner 3)--(a.corner 4)--(a.corner 5)--(a.corner 6)--(a.corner 1);
\node[regular polygon,regular polygon sides=6,minimum size=1cm,draw] (a){};

\foreach \x in {1,...,6}{\node[circle,fill,inner sep=1.5pt] at (a.corner \x) {};}
\foreach \x in {1,...,6}
{\draw (a.corner 2) -- (a.corner \x) {};}
\filldraw [white] (a.corner 2) circle (1.pt);
\end{tikzpicture}

$F_5$
\end{center}
\end{minipage}

\end{center}

Below are the subsequent steps of attaching $F_j$.

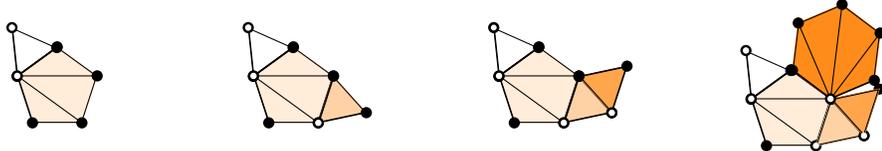
\begin{minipage}{.2\textwidth}
\begin{center}
\begin{tikzpicture}
\node[regular polygon,regular polygon sides=5,minimum size=1.1cm,draw] (a){};
\filldraw[color=orange!15] (a.corner 1)--(a.corner 2)--(a.corner 3)--(a.corner 4)--(a.corner 5)--(a.corner 1);
\node[regular polygon,regular polygon sides=5,minimum size=1.1cm,draw] (a){};
\foreach \x in {1,...,5}{\node[circle,fill,inner sep=1.5pt] at (a.corner \x) {};}
\foreach \x in {1,...,5}
{\draw (a.corner 2) -- (a.corner \x) {};}

\node[regular polygon,regular polygon sides=3,minimum size=.3cm,draw,rotate=36] at (-.38,.51) (b){};
\foreach \x in {1,...,3}{\node[circle,fill,inner sep=1.5pt] at (b.corner \x) {};}
\foreach \x in {1,...,3}
{\draw (b.corner 2) -- (b.corner \x) {};}
\filldraw [white] (a.corner 2) circle (1.pt);
\filldraw [white] (b.corner 1) circle (1.pt);
\end{tikzpicture}

\end{center}
\end{minipage}
\begin{minipage}{.2\textwidth}
\begin{center}
\begin{tikzpicture}
\node[regular polygon,regular polygon sides=5,minimum size=1.1cm,draw] (b){};
\filldraw[color=orange!15] (b.corner 1)--(b.corner 2)--(b.corner 3)--(b.corner 4)--(b.corner 5)--(b.corner 1);
\node[regular polygon,regular polygon sides=5,minimum size=1.1cm,draw] (b){};
\foreach \x in {1,...,5}{\node[circle,fill,inner sep=1.5pt] at (a.corner \x) {};}
\foreach \x in {1,...,5}
{\draw (b.corner 2) -- (b.corner \x) {};}

\node[regular polygon,regular polygon sides=3,minimum size=.3cm,draw,rotate=36] at (-.38,.51) (a){};
\foreach \x in {1,...,3}{\node[circle,fill,inner sep=1.5pt] at (a.corner \x) {};}
\foreach \x in {1,...,3}
{\draw (a.corner 2) -- (a.corner \x) {};}

\node[regular polygon,regular polygon sides=3,minimum size=.3cm,draw,rotate=12] at (.61,-.2) (c){};
\filldraw[color=orange!35] (c.corner 1)--(c.corner 2)--(c.corner 3)--(c.corner 1);
\node[regular polygon,regular polygon sides=3,minimum size=.3cm,draw,rotate=12] at (.61,-.2) (c){};

\foreach \x in {1,...,3}{\node[circle,fill,inner sep=1.5pt] at (c.corner \x) {};}
\foreach \x in {1,...,3}
{\draw (c.corner 2) -- (c.corner \x) {};}

\filldraw [white] (a.corner 1) circle (1.pt);
\filldraw [white] (b.corner 2) circle (1.pt);
\filldraw [white] (c.corner 2) circle (1.pt);

\end{tikzpicture}

\end{center}
\end{minipage}
\begin{minipage}{.2\textwidth}
\begin{center}
\begin{tikzpicture}
\node[regular polygon,regular polygon sides=5,minimum size=1.1cm,draw] (b){};
\filldraw[color=orange!15] (b.corner 1)--(b.corner 2)--(b.corner 3)--(b.corner 4)--(b.corner 5)--(b.corner 1);
\node[regular polygon,regular polygon sides=5,minimum size=1.1cm,draw] (b){};
\foreach \x in {1,...,5}{\node[circle,fill,inner sep=1.5pt] at (b.corner \x) {};}
\foreach \x in {1,...,5}
{\draw (b.corner 2) -- (b.corner \x) {};}

\node[regular polygon,regular polygon sides=3,minimum size=.3cm,draw,rotate=36] at (-.38,.51) (a){};
\foreach \x in {1,...,3}{\node[circle,fill,inner sep=1.5pt] at (a.corner \x) {};}
\foreach \x in {1,...,3}
{\draw (a.corner 2) -- (a.corner \x) {};}

\node[regular polygon,regular polygon sides=3,minimum size=.3cm,draw,rotate=12] at (.61,-.2) (c){};
\node[regular polygon,regular polygon sides=3,minimum size=.3cm,draw,rotate=12] at (.61,-.2) (c){};
\filldraw[color=orange!35] (c.corner 1)--(c.corner 2)--(c.corner 3)--(c.corner 1);
\node[regular polygon,regular polygon sides=3,minimum size=.3cm,draw,rotate=12] at (.61,-.2) (c){};
\foreach \x in {1,...,3}{\node[circle,fill,inner sep=1.5pt] at (c.corner \x) {};}
\foreach \x in {1,...,3}
{\draw (c.corner 2) -- (c.corner \x) {};}

\node[regular polygon,regular polygon sides=3,minimum size=.3cm,draw,rotate=192] at (.89,.05) (d){};
\filldraw[color=orange!70] (d.corner 1)--(d.corner 2)--(d.corner 3)--(d.corner 1);
\node[regular polygon,regular polygon sides=3,minimum size=.3cm,draw,rotate=192] at (.89,.05) (d){};
\foreach \x in {1,...,3}{\node[circle,fill,inner sep=1.5pt] at (d.corner \x) {};}
\foreach \x in {1,...,3}
{\draw (d.corner 2) -- (d.corner \x) {};}

\filldraw [white] (a.corner 1) circle (.9pt);
\filldraw [white] (b.corner 2) circle (1.pt);
\filldraw [white] (c.corner 2) circle (1.pt);
\filldraw [white] (d.corner 1) circle (1.pt);
\end{tikzpicture}

\end{center}
\end{minipage}
\begin{minipage}{.2\textwidth}
\begin{center}
\begin{tikzpicture}
\node[regular polygon,regular polygon sides=5,minimum size=1.1cm,draw] (b){};
\filldraw[color=orange!15] (b.corner 1)--(b.corner 2)--(b.corner 3)--(b.corner 4)--(b.corner 5)--(b.corner 1);
\node[regular polygon,regular polygon sides=5,minimum size=1.1cm,draw] (b){};
\foreach \x in {1,...,5}{\node[circle,fill,inner sep=1.5pt] at (b.corner \x) {};}
\foreach \x in {1,...,5}
{\draw (b.corner 2) -- (b.corner \x) {};}

\node[regular polygon,regular polygon sides=3,minimum size=.3cm,draw,rotate=36] at (-.38,.51) (a){};
\foreach \x in {1,...,3}{\node[circle,fill,inner sep=1.5pt] at (a.corner \x) {};}
\foreach \x in {1,...,3}
{\draw (a.corner 2) -- (a.corner \x) {};}

\node[regular polygon,regular polygon sides=3,minimum size=.3cm,draw,rotate=12] at (.61,-.2) (c){};
\foreach \x in {1,...,3}{\node[circle,fill,inner sep=1.5pt] at (c.corner \x) {};}
\foreach \x in {1,...,3}
{\draw (c.corner 2) -- (c.corner \x) {};}
\node[regular polygon,regular polygon sides=3,minimum size=.3cm,draw,rotate=12] at (.61,-.2) (c){};
\filldraw[color=orange!35] (c.corner 1)--(c.corner 2)--(c.corner 3)--(c.corner 1);
\node[regular polygon,regular polygon sides=3,minimum size=.3cm,draw,rotate=12] at (.61,-.2) (c){};

\node[regular polygon,regular polygon sides=3,minimum size=.3cm,draw,rotate=192] at (.89,.05) (d){};
\foreach \x in {1,...,3}{\node[circle,fill,inner sep=1.5pt] at (d.corner \x) {};}
\foreach \x in {1,...,3}
{\draw (d.corner 2) -- (d.corner \x) {};}

\filldraw[color=orange!70] (d.corner 1)--(d.corner 2)--(d.corner 3)--(d.corner 1);
\node[regular polygon,regular polygon sides=3,minimum size=.3cm,draw,rotate=192] at (.89,.05) (d){};

\node[regular polygon,regular polygon sides=6,minimum size=1.25cm,draw,rotate=143] at (.6,.8) (e){};

\filldraw[color=orange!90] (e.corner 1)--(e.corner 2)--(e.corner 3)--(e.corner 4)--(e.corner 5)--(e.corner 6)--(e.corner 1);
\node[regular polygon,regular polygon sides=6,minimum size=1.25cm,draw,rotate=143] at (.6,.8) (e){};
\foreach \x in {1,...,6}{\node[circle,fill,inner sep=1.5pt] at (e.corner \x) {};}
\foreach \x in {1,...,6}
{\draw (e.corner 2) -- (e.corner \x) {};}

\filldraw [white] (a.corner 1) circle (1.pt);
\filldraw [white] (b.corner 2) circle (1.pt);
\filldraw [white] (c.corner 2) circle (1.pt);
\filldraw [white] (d.corner 1) circle (1.pt);
\filldraw [white] (e.corner 2) circle (1.pt);
\end{tikzpicture}

\end{center}
\end{minipage}

Redrawn with vertex labels, we have the following, where the vertex labeled 1 is the origin of $F_1$, the vertex labeled 2 is the origin of $F_2$, and so on.
\begin{center}

\begin{tikzpicture}[scale=1]
\node[regular polygon,regular polygon sides=12,minimum size=4cm,draw,rotate=30] at (.74,.76) (e){};

\filldraw[color=orange!1] (e.corner 1)--(e.corner 2)--(e.corner 12)--(e.corner 1);
\draw[thick] (e.corner 1)--(e.corner 2)--(e.corner 12)--(e.corner 1);

\filldraw[color=orange!15] (e.corner 2)--(e.corner 12)--(e.corner 7)--(e.corner 4)--(e.corner 3)--(e.corner 2);
\draw[thick] (e.corner 2)--(e.corner 12)--(e.corner 7)--(e.corner 4)--(e.corner 3)--(e.corner 2);

\filldraw[color=orange!35] (e.corner 4)--(e.corner 5)--(e.corner 7)--(e.corner 4);
\draw[thick] (e.corner 4)--(e.corner 5)--(e.corner 7)--(e.corner 4);

\filldraw[color=orange!70] (e.corner 6)--(e.corner 5)--(e.corner 7)--(e.corner 6);
\draw[thick] (e.corner 6)--(e.corner 5)--(e.corner 7)--(e.corner 6);

\filldraw[color=orange!100] (e.corner 7)--(e.corner 8)--(e.corner 9)--(e.corner 10)--(e.corner 11)--(e.corner 12)--(e.corner 7);
\draw[thick] (e.corner 7)--(e.corner 8)--(e.corner 9)--(e.corner 10)--(e.corner 11)--(e.corner 12)--(e.corner 7);

\node[regular polygon,regular polygon sides=12,minimum size=4cm,draw,rotate=30] at (.74,.76) (e){};

\foreach \x in {1,...,12}{\node[circle,fill,inner sep=1.5pt] at (e.corner \x) {};}

\node[above] at (e.corner 1) {$1$};
\node[above] at (e.corner 2) {$2$};
\node[left] at (e.corner 3) {$3$};
\node[left] at (e.corner 4) {$4$};
\node[below] at (e.corner 5) {$5$};
\node[below] at (e.corner 6) {$6$};
\node[below] at (e.corner 7) {$7$};
\node[below] at (e.corner 8) {$8$};
\node[right] at (e.corner 9) {$9$};
\node[right] at (e.corner 10) {$10$};
\node[above] at (e.corner 11) {$11$};
\node[above] at (e.corner 12) {$12$};

{\draw (e.corner 2) -- (e.corner 12) {};}
{\draw (e.corner 2) -- (e.corner 7) {};}
{\draw (e.corner 2) -- (e.corner 4) {};}
{\draw (e.corner 4) -- (e.corner 7) {};}
{\draw (e.corner 5) -- (e.corner 7) {};}
{\draw (e.corner 6) -- (e.corner 7) {};}
{\draw (e.corner 7) -- (e.corner 10) {};}
{\draw (e.corner 7) -- (e.corner 9) {};}
{\draw (e.corner 7) -- (e.corner 11) {};}
{\draw (e.corner 7) -- (e.corner 12) {};}

\filldraw [white] (e.corner 1) circle (1.25pt);
\filldraw [white] (e.corner 2) circle (1.25pt);
\filldraw [white] (e.corner 4) circle (1.25pt);
\filldraw [white] (e.corner 5) circle (1.25pt);
\filldraw [white] (e.corner 7) circle (1.25pt);
\end{tikzpicture}
\end{center}

\end{example}

\subsection{Map $\td$}\label{map_td}
In this section, we construct a map $\td$ from $\Tri(n,t,s)$ to $\Dyck(n-2,t,s)$\editone{, which is the inverse to $\dt$.}

\begin{Def}
%Given $T\in \Tri(n,t,s)$, we study $\cF(T)=(F_1,F_2,\dots, F_{t+s})$, where $t$ is the number of non-singular fans in $\cF(T)$ and $s$ is the number of singular fans in $\cF(T)$. We let $x_j$ denote the origin of $F_j$. Now we identify the $x_j$ with its corresponding vertex in $T$.  Starting at $x_1$, and traveling counter-clockwise around the boundary of $T$, we let $d_j$ denote the number of boundary edges between $x_j$ and $x_{j+1}$ for $1\leq j\leq t+s-1$. We let $d_{t+s}$ be the number of boundary edges from $x_{t+s}$ to $x_1$ minus 2. Letting $u_j$ be the number of triangles in $F_j$, define $\td(T)$ to be
%\[\U^{u_1}\D^{d_1}\U^{u_2}\D^{d_2}\dots \U^{u_{t+s}}\D^{d_{t+s}}.\]

Given $T\in \Tri(n,t,s)$, we define $\td(T)$, a certain lattice path, as follows. 

Consider the fan decomposition $(\cF(T), \d(T)) = ((F_1,F_2,\ldots, F_{t+s}), (d_1,\ldots,d_{t+s-1})$. Let $x_j$ denote the origin of $F_j$ and let $d_{t+s}$ be the number of boundary edges from $x_{t+s}$ to $x_1$ minus 2. Now we identify the $x_j$ with its corresponding vertex in $T$. Letting $u_j$ be the number of triangles in $F_j$, define $\td(T)$ to be
\[\U^{u_1}\D^{d_1}\U^{u_2}\D^{d_2}\dots \U^{u_{t+s}}\D^{d_{t+s}}.\]

\end{Def}

\begin{lemma}
If $T\in \Tri(n,t,s)$, then the string $\td(T)$ is a dyck path in $\Dyck(n-2,t,s)$.
\end{lemma}
\begin{proof}
%\GDN{Double check the following, including what shows up at the top of the next page.}
We claim $\td(T)$ is a valid Dyck path. First, note that given any fan decomposition of a triangulation on $n$ vertices, the largest label for an origin is $n-1$. Thus, the largest  number of boundary edges between the first and last origin (travelling counter-clockwise) is $n-2$, which is precisely the number of triangles in the triangulation. 
Thus, by studying the triangulation $T_k$ given by $\mathcal{F}(T_k)=(F_1,F_2,\dots, F_k)$ and the origins $x_1,\dots, x_k$, we see that 
\[\sum_{j=1}^k d_j\leq \text{the number of triangles in $T_k$}=\sum_{j=1}^k u_j.\]
Also, note that
\begin{align*}
\displaystyle \sum_{j=1}^{t+s}u_j&=\text{number of triangles in $T$}\\
&=n-2\\
&=\text{number of boundary edges of $T$ minus 2}\\
&=\sum_{j=1}^{t+s}d_j.
\end{align*}  Hence this path is a Dyck path with semi length $n-2$. Finally, note that $u_{t+s}=\#F_{t+s}\geq 2$ since $F_{t+s}$ is the last fan in $\mathcal{F}(T)$. Hence, we have constructed a path in $\Dyck(n-2,t,s)$.
\end{proof}

\begin{example}
Suppose we start with the following triangulation. 

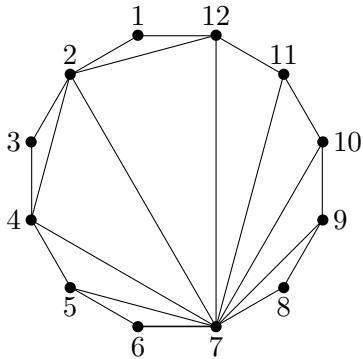
\begin{center}
\begin{tikzpicture}[scale=1]
\node[regular polygon,regular polygon sides=12,minimum size=4cm,draw,rotate=30] at (.74,.76) (e){};
\foreach \x in {1,...,12}{\node[circle,fill,inner sep=1.5pt] at (e.corner \x) {};}

\node[above] at (e.corner 1) {$1$};
\node[above] at (e.corner 2) {$2$};
\node[left] at (e.corner 3) {$3$};
\node[left] at (e.corner 4) {$4$};
\node[below] at (e.corner 5) {$5$};
\node[below] at (e.corner 6) {$6$};
\node[below] at (e.corner 7) {$7$};
\node[below] at (e.corner 8) {$8$};
\node[right] at (e.corner 9) {$9$};
\node[right] at (e.corner 10) {$10$};
\node[above] at (e.corner 11) {$11$};
\node[above] at (e.corner 12) {$12$};

{\draw (e.corner 2) -- (e.corner 12) {};}
{\draw (e.corner 2) -- (e.corner 7) {};}
{\draw (e.corner 2) -- (e.corner 4) {};}
{\draw (e.corner 4) -- (e.corner 7) {};}
{\draw (e.corner 5) -- (e.corner 7) {};}
{\draw (e.corner 6) -- (e.corner 7) {};}
{\draw (e.corner 7) -- (e.corner 10) {};}
{\draw (e.corner 7) -- (e.corner 9) {};}
{\draw (e.corner 7) -- (e.corner 11) {};}
{\draw (e.corner 7) -- (e.corner 12) {};}
\end{tikzpicture}
\end{center}

Below, we identify the origins of the fans in $\mathcal{F}=(F_1,F_2,F_3,F_4,F_5)$ by using larger circles for such vertices. 
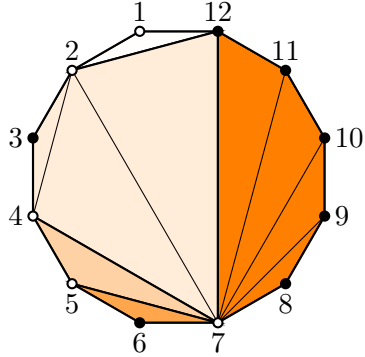
\begin{center}
\begin{tikzpicture}[scale=.9]
\node[regular polygon,regular polygon sides=12,minimum size=4cm,draw,rotate=30] at (.74,.76) (e){};

\filldraw[color=orange!1] (e.corner 1)--(e.corner 2)--(e.corner 12)--(e.corner 1);
\draw[thick] (e.corner 1)--(e.corner 2)--(e.corner 12)--(e.corner 1);

\filldraw[color=orange!15] (e.corner 2)--(e.corner 12)--(e.corner 7)--(e.corner 4)--(e.corner 3)--(e.corner 2);
\draw[thick] (e.corner 2)--(e.corner 12)--(e.corner 7)--(e.corner 4)--(e.corner 3)--(e.corner 2);

\filldraw[color=orange!35] (e.corner 4)--(e.corner 5)--(e.corner 7)--(e.corner 4);
\draw[thick] (e.corner 4)--(e.corner 5)--(e.corner 7)--(e.corner 4);

\filldraw[color=orange!70] (e.corner 6)--(e.corner 5)--(e.corner 7)--(e.corner 6);
\draw[thick] (e.corner 6)--(e.corner 5)--(e.corner 7)--(e.corner 6);

\filldraw[color=orange!100] (e.corner 7)--(e.corner 8)--(e.corner 9)--(e.corner 10)--(e.corner 11)--(e.corner 12)--(e.corner 7);
\draw[thick] (e.corner 7)--(e.corner 8)--(e.corner 9)--(e.corner 10)--(e.corner 11)--(e.corner 12)--(e.corner 7);

\node[regular polygon,regular polygon sides=12,minimum size=4cm,draw,rotate=30] at (.74,.76) (e){};

\node[above] at (e.corner 1) {$1$};
\node[above] at (e.corner 2) {$2$};
\node[left] at (e.corner 3) {$3$};
\node[left] at (e.corner 4) {$4$};
\node[below] at (e.corner 5) {$5$};
\node[below] at (e.corner 6) {$6$};
\node[below] at (e.corner 7) {$7$};
\node[below] at (e.corner 8) {$8$};
\node[right] at (e.corner 9) {$9$};
\node[right] at (e.corner 10) {$10$};
\node[above] at (e.corner 11) {$11$};
\node[above] at (e.corner 12) {$12$};

\foreach \x in {1,...,12}{\node[circle,fill,inner sep=1.5pt] at (e.corner \x) {};}

{\draw (e.corner 2) -- (e.corner 12) {};}
{\draw (e.corner 2) -- (e.corner 7) {};}
{\draw (e.corner 2) -- (e.corner 4) {};}
{\draw (e.corner 4) -- (e.corner 7) {};}
{\draw (e.corner 5) -- (e.corner 7) {};}
{\draw (e.corner 6) -- (e.corner 7) {};}
{\draw (e.corner 7) -- (e.corner 10) {};}
{\draw (e.corner 7) -- (e.corner 9) {};}
{\draw (e.corner 7) -- (e.corner 11) {};}
{\draw (e.corner 7) -- (e.corner 12) {};}

\filldraw [white] (e.corner 1) circle (1.25pt);
\filldraw [white] (e.corner 2) circle (1.25pt);
\filldraw [white] (e.corner 4) circle (1.25pt);
\filldraw [white] (e.corner 5) circle (1.25pt);
\filldraw [white] (e.corner 7) circle (1.25pt);
\end{tikzpicture}
\end{center}

 Hence,  $F_1$, $F_3$, and $F_4$ are all fans of size 1, $F_2$ is a fan of size 3, and $F_5$ is a fan of size 4. The origins of these fans are identified in the image below. Consequently, $u_1=u_3=u_4=1$, $u_2=3$, and $u_5=4$. Also note that $d_1=d_3$ and $d_2=d_4=2$. Hence we have the Dyck path 
 \[\U\D\U^3\D^2\U\D\U\D^2\U^4\D^4,\]
 which visualized gives the following path.
 
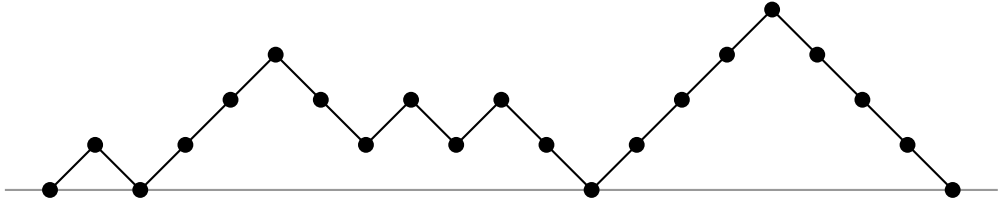
\begin{center}
\begin{tikzpicture}[scale=0.6, line width=1pt]
 
\draw[color=black!40, thick] (-3,0)--(19,0); 
 
 \fill (-2,0) circle[radius=5pt];
 \draw[thick] (-2,0)--(-1,1);
 \fill (-1,1) circle[radius=5pt];
 \draw[thick] (-1,1)--(0,0);
 \fill (0,0) circle[radius=5pt];
 
 \foreach \x in {0,1,...,2}{
	\draw[thick] (\x,\x)--(\x+1,\x+1);
  	\fill (\x+1,\x+1) circle[radius=5pt];}
  	
 \draw[thick] (3,3)--(4,2);
 \fill (4,2) circle[radius=5pt];
 \draw[thick] (4,2)--(5,1);
 \fill (5,1) circle[radius=5pt];
 \draw[thick] (5,1)--(6,2);
 \fill (6,2) circle[radius=5pt];
 \draw[thick] (6,2)--(7,1);
 \fill (7,1) circle[radius=5pt];
 \draw[thick] (7,1)--(8,2);
 \fill (8,2) circle[radius=5pt];
 \draw[thick] (8,2)--(9,1);
 \fill (9,1) circle[radius=5pt];
 \draw[thick] (9,1)--(10,0);
 \fill (10,0) circle[radius=5pt];
 \draw[thick] (10,0)--(14,4);
 \fill (11,1) circle[radius=5pt];
 \fill (12,2) circle[radius=5pt];
 \fill (13,3) circle[radius=5pt];
 \fill (14,4) circle[radius=5pt];
 \draw[thick] (14,4)--(18,0);
 \fill (15,3) circle[radius=5pt];
 \fill (16,2) circle[radius=5pt];
 \fill (17,1) circle[radius=5pt];
 \fill (18,0) circle[radius=5pt];

\end{tikzpicture}
\end{center}

The following map is an explicit interpretation of the map $\sd$ composed with $\dt$ without needing to bring up Dyck paths.
\end{example}

\subsection{Map $\st$}\label{map_st}
In this section, we construct a map $\st$ from $\Skyt(a,i,b)$ to $\Tri(a+b+2i,i+1,b-2)$.

\begin{Def}
Given a tableau $\lambda\in \Skyt(a,i,b)$\editone{, we define $\st(\lambda)$, a triangulation, as follows. }

%Let $\l_1$ be the skew tableau obtained by taking the top $b-1$ rows of $\l$.
% and $\l_2$ be the skew tableau obtained by taking the bottom $a-1$ rows of $\l$. 
%\GDN{I'm getting rid of the reference to $\l_2$ because this map doesn't use it.}
%Observe that via the fan decomposition we gave above, a triangulation $T$ is completely determined by a sequence of fans $\cF(T)$ and the number of boundary edges between origins of consecutive fans in $\cF(T)$. 
\editone{Let $d_j=x_{j+1}-x_j$, where $x_1,\dots, x_{i+b-1}$ are the entries in the top $b-1$ rows of $\l$ so that $x_1<x_2<\dots<x_{i+b-1}$. 
\edittwo{Let $\Nom(\lambda)=(A_1,\dots, A_{i+b-1})$.} For $1 \leq j\leq i+b-1$, let $f_j=\#A_j$.
Let $F_j$ be a fan of size $f_j$.  Then we define $\st(\lambda)$ to be the triangulation whose fan decomposition is $(\mathcal{F},\delta)$, where $\mathcal{F}=(F_1,\dots, F_{i+b-1})$ and $\delta=(d_1,\dots, d_{i+b-2})$.}

\end{Def}

\begin{lemma}
Let $\lambda\in \Skyt(a,i,b)$. Then $\st(\lambda)\in \Tri(a+b+2i,i+1,b-2)$.
\end{lemma}  
\begin{proof}
Recall triangulations are uniquely determined by their fan decomposition, and thus $\st(\lambda)$ is guaranteed to be a triangulation.
Note that the number of triangles in this triangulation is precisely the number of entries of $\lambda$, which is $a+b+2i-2$. Hence, the boundary of our constructed triangulation has $a+b+2i$ edges. Recall that among $A_1,\dots, A_{i+b-1}$, precisely $i+1$ have cardinality larger than 1, and precisely $b-2$ have cardinality exactly 1. Thus, our proposed fan decomposition for $\st(\lambda)$ has precisely $i+1$ non singular fans and $b-2$ singular fans. Finally, note that by construction, $f_{i+b-1}>1$ since $|A_{i+b-1}|>1$ \editone{by construction of $\Nom(\lambda)$. T}hat is, the last fan appearing in $\mathcal{F}$ is not singular. All together, this verifies that we have constructed a triangulation in $\Tri(a+b+2i, i+1, b-2)$.
\end{proof}

\begin{example}
Consider the following choice for $\lambda$. 
%The corresponding $\lambda_1$ and $\lambda_2$ are also given.

\usetikzlibrary{decorations.pathreplacing}
\newcommand{\tikznode}[3][inner sep=0pt]{\tikz[remember
picture,baseline=(#2.base)]{\node(#2)[#1]{$#3$};}}
\ytableausetup{centertableaux}
\begin{center}

\begin{ytableau}
       \none & \none & 1 \\
       \none & \none & 4 \\
       2 &5 & 7 \\
       3&6&9\\
       8&\none&\none\\
              10&\none&\none\\
\end{ytableau}
\end{center}

\editone{Hence, $A_1=\{1\}$, $A_2=\{2,3\}$, $A_3=\{4\}$, $A_4=\{5,6,8\}$, and $A_5=\{7,9,10\}$. We have $x_1=1$, $x_2=2$, $x_3=4$, $x_4=5$, and $x_5=7$. Thus, $d_1=1$, $d_2=2$, $d_3=1$, and $d_4=2$. Also $f_1=1$, $f_2=2$, $f_3=1$, $f_4=3$, and $f_5=3$. This constructs the following triangulation.}

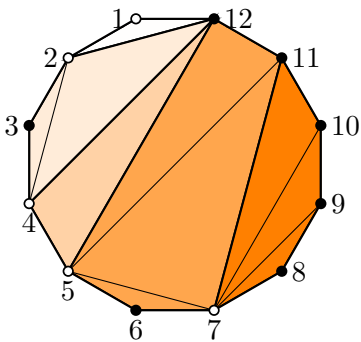
\begin{center}
\begin{tikzpicture}[scale=1]
\node[regular polygon,regular polygon sides=12,minimum size=4cm,draw,rotate=30] at (.74,.76) (e){};

\filldraw[color=orange!1] (e.corner 1)--(e.corner 2)--(e.corner 12)--(e.corner 1);

\draw[thick] (e.corner 1)--(e.corner 2)--(e.corner 12)--(e.corner 1);

\filldraw[color=orange!15] (e.corner 2)--(e.corner 12)--(e.corner 4)--(e.corner 3)--(e.corner 2);
\draw[thick] (e.corner 2)--(e.corner 12)--(e.corner 4)--(e.corner 3)--(e.corner 2);

\filldraw[color=orange!35] (e.corner 4)--(e.corner 5)--(e.corner 12)--(e.corner 4);
\draw[thick] (e.corner 4)--(e.corner 5)--(e.corner 12)--(e.corner 4);

\filldraw[color=orange!70] (e.corner 6)--(e.corner 5)--(e.corner 12)--(e.corner 11)--(e.corner 7)--(e.corner 6);
\draw[thick] (e.corner 6)--(e.corner 5)--(e.corner 12)--(e.corner 11)--(e.corner 7)--(e.corner 6);

\filldraw[color=orange!100] (e.corner 7)--(e.corner 8)--(e.corner 9)--(e.corner 10)--(e.corner 11)--(e.corner 7);
\draw[thick] (e.corner 7)--(e.corner 8)--(e.corner 9)--(e.corner 10)--(e.corner 11)--(e.corner 7);

\node[regular polygon,regular polygon sides=12,minimum size=4cm,draw,rotate=30] at (.74,.76) (e){};

\foreach \x in {1,...,12}{\node[circle,fill,inner sep=1.5pt] at (e.corner \x) {};}

\foreach \x in {1,...,3}{\node[left] at (e.corner \x) {\x};}

\foreach \x in {4,...,7}{\node[below] at (e.corner \x) {\x};}

\foreach \x in {8,...,12}{\node[right] at (e.corner \x) {\x};}

{\draw (e.corner 2) -- (e.corner 12) {};}
{\draw (e.corner 2) -- (e.corner 4) {};}
{\draw (e.corner 4) -- (e.corner 12) {};}
{\draw (e.corner 5) -- (e.corner 12) {};}
{\draw (e.corner 5) -- (e.corner 7) {};}
{\draw (e.corner 6) -- (e.corner 7) {};}
{\draw (e.corner 7) -- (e.corner 10) {};}
{\draw (e.corner 7) -- (e.corner 9) {};}
{\draw (e.corner 5) -- (e.corner 11) {};}
{\draw (e.corner 7) -- (e.corner 11) {};}

\filldraw [white] (e.corner 1) circle (1.25pt);
\filldraw [white] (e.corner 2) circle (1.25pt);
\filldraw [white] (e.corner 4) circle (1.25pt);
\filldraw [white] (e.corner 5) circle (1.25pt);
\filldraw [white] (e.corner 7) circle (1.25pt);
\end{tikzpicture}
\end{center}
\end{example}

%Now suppose that we are given a tableau $\lambda\in \Skyt(a,i,b)$. Let $\l_1$ be the skew tableau obtained by taking the top $b-1$ rows of $\l$ and $\l_2$ be the skew tableau obtained by taking the bottom $a-1$ rows of $\l$. Observe that via the fan decomposition we gave above, a triangulation $T$ is completely determined by a sequence of fans $\cF(T)$ and the number of boundary edges between origins of consecutive fans in $\cF(T)$. Let $\beta_j=v_{j+1}-v_j$, where $v_1,\dots, v_{i+b-1}$ are the entries of $\l_1$ so that $v_1<v_2<\dots, v_{i+b-1}$. Let $(m_k)$ be a sequence so that $v_{m_1},\ldots, v_{m_{i+1}}$ are the entries of row $b-1$ so that $v_{m_1}<v_{m_2}<\cdots < v_{m_{i+1}}$.  
%
%Let $w_1,\dots, w_{a+i-1}$ be the entries of $\l_2$ so that $w_1<w_2<\dots<w_{a+i-1}$. Let $(\ell_{k})$ be a sequence such that $w_{\ell_1},w_{\ell_2},\dots, w_{\ell_{i+1}}$ are the entries of row 1 of $\l_2$ and $w_{\ell_1}<w_{\ell_2}<\cdots < w_{\ell_{i+1}}$. For $1\leq j\leq i+b-1$, let
%\[f_j\coloneqq\begin{cases}
%\ell_{k+1}-\ell_{k}+1 & j=m_k \text{ for some $k$}\\
%1 & \text{otherwise}.
%\end{cases}\]
%
%Let $F_j$ be a fan of size $f_j$.  Then we define $T$ to be the triangulation so that $\cF(T)=\{F_1,\dots, F_{i+b-1}\}$ where $\beta_j$ is the number of boundary edges between the origins of $F_j$ and $F_{j+1}$.

The following map is an explicit interpretation of the map $\td$ composed with $\ds$ without needing to bring up Dyck paths.

\subsection{Map $\ts$}\label{map_ts}
%Given $T\in \Tri(n,t,s)$, we study $\cF(T)=(F_1,F_2,\dots, F_{s+t})$. We define $e_j$ to be the boundary edge of $F_j$ containing the origin of $F_j$, which we denote $x_j$. Starting at $x_1$, and traveling counter-clockwise around the boundary of $T$, we let $\b_j$ denote the number of boundary edges between $x_j$ and $x_{j+1}$ for $1\leq j\leq t+s$. We let $\b_{t+s}$ be the number of boundary edges from $x_{t+s}$ to $x_1$ minus 2. Label the vertices in counter-clockwise order starting at $x_1$.  
In this section, we construct a map $\td$ from $\Tri(n,t,s)$ to $\Skyt(n-2-s-2t,t-1,s+2)$\editone{, which is the inverse to $\st$.}

\begin{Def}
Let $T\in \Tri(n,t,s)$. We define $\ts(T)$, a skew tableau, as follows. 

 We label the triangles in the triangulation in the following way: For each fan, label a triangle with the label of the fan's origin. Then, greedily label other triangles with the unused vertices in the order that fans appear in $\mathcal{F}(T)$. (Triangles within a single fan need not be labeled in any particular order.) Let $A_j$ be the labels appearing in the fan $F_j$. 
Let $\ts(T)$ be the skew diagram \editone{so that $\Nom\big(\ts(T)\big)=(A_1,\dots, A_{t+s})$.}
\end{Def}
\begin{lemma}
Given $T\in \Tri(n,t,s)$, we have $\ts(T) \in \Skyt(n-2-s-2t,t-1,s+2)$.
\end{lemma}
\begin{proof}
Note that the number of $A_j$ so that $|A_j|=1$ is precisely the number of singleton fans in $T$, which is $s$. 
%Thus the last column of $\ts(T)$ has $s+2$ entries. 
Also, the number of $A_j$ so that $|A_j|>1$ is $t$. 
%Thus the number of columns in $\ts(T)$ is $t$. 
\editone{Next, notice that $(A_1,\dots, A_{\ell+s})$ is a nomincreasing sequence due to the greedy labeling of the triangles in $T$, and we also have that $|A_{\ell+s}|>1$ since $F_{\ell+s}$, the last fan appearing in $\mathcal{F}(T)$, must contain more than 1 triangle.
Finally, the number of triangles in $T$ is $n-2$.} Thus, $\ts(T)\in \Skyt(n-2-s-2t,t-1,s+2)$.
\end{proof}

\begin{example}
For example, given the triangulation in $\Tri(12,3,2)$ below, label the vertices in counter-clockwise order. 

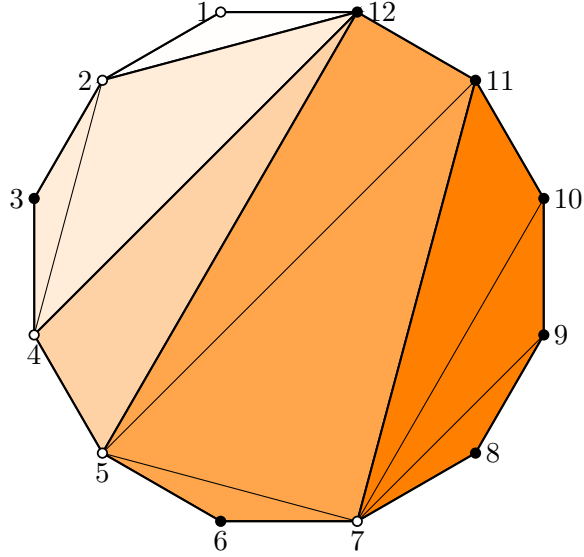
\begin{center}

\begin{tikzpicture}[scale=1]
\node[regular polygon,regular polygon sides=12,minimum size=7cm,draw,rotate=30] at (.74,.76) (e){};

\filldraw[color=orange!1] (e.corner 1)--(e.corner 2)--(e.corner 12)--(e.corner 1);

\draw[thick] (e.corner 1)--(e.corner 2)--(e.corner 12)--(e.corner 1);

\filldraw[color=orange!15] (e.corner 2)--(e.corner 12)--(e.corner 4)--(e.corner 3)--(e.corner 2);
\draw[thick] (e.corner 2)--(e.corner 12)--(e.corner 4)--(e.corner 3)--(e.corner 2);

\filldraw[color=orange!35] (e.corner 4)--(e.corner 5)--(e.corner 12)--(e.corner 4);
\draw[thick] (e.corner 4)--(e.corner 5)--(e.corner 12)--(e.corner 4);

\filldraw[color=orange!70] (e.corner 6)--(e.corner 5)--(e.corner 12)--(e.corner 11)--(e.corner 7)--(e.corner 6);
\draw[thick] (e.corner 6)--(e.corner 5)--(e.corner 12)--(e.corner 11)--(e.corner 7)--(e.corner 6);

\filldraw[color=orange!100] (e.corner 7)--(e.corner 8)--(e.corner 9)--(e.corner 10)--(e.corner 11)--(e.corner 7);
\draw[thick] (e.corner 7)--(e.corner 8)--(e.corner 9)--(e.corner 10)--(e.corner 11)--(e.corner 7);

\node[regular polygon,regular polygon sides=12,minimum size=7cm,draw,rotate=30] at (.74,.76) (e){};

\foreach \x in {1,...,12}{\node[circle,fill,inner sep=1.5pt] at (e.corner \x) {};}

\foreach \x in {1,...,3}{\node[left] at (e.corner \x) {\x};}

\foreach \x in {4,...,7}{\node[below] at (e.corner \x) {\x};}

\foreach \x in {8,...,12}{\node[right] at (e.corner \x) {\x};}

{\draw (e.corner 2) -- (e.corner 12) {};}
{\draw (e.corner 2) -- (e.corner 4) {};}
{\draw (e.corner 4) -- (e.corner 12) {};}
{\draw (e.corner 5) -- (e.corner 12) {};}
{\draw (e.corner 5) -- (e.corner 7) {};}
{\draw (e.corner 6) -- (e.corner 7) {};}
{\draw (e.corner 7) -- (e.corner 10) {};}
{\draw (e.corner 7) -- (e.corner 9) {};}
{\draw (e.corner 5) -- (e.corner 11) {};}
{\draw (e.corner 7) -- (e.corner 11) {};}

\filldraw [white] (e.corner 1) circle (1.25pt);
%\node[above] at (e.corner 1) {$x_1$};
\filldraw [white] (e.corner 2) circle (1.25pt);
%\node[above] at (e.corner 2) {$x_2$};
\filldraw [white] (e.corner 4) circle (1.25pt);
%\node[left] at (e.corner 4) {$x_3$};
\filldraw [white] (e.corner 5) circle (1.25pt);
%\node[left] at (e.corner 5) {$x_4$};
\filldraw [white] (e.corner 7) circle (1.25pt);
%\node[below] at (e.corner 7) {$x_5$};
\end{tikzpicture}

\end{center}
We now label a single triangle in each fan with the label of the corresponding origin. 

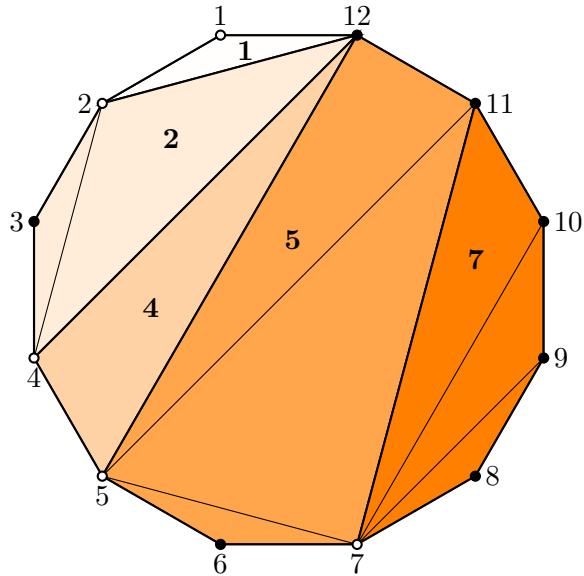
\begin{center}

\begin{tikzpicture}[scale=.9]
\node[regular polygon,regular polygon sides=12,minimum size=7cm,draw,rotate=30] at (.74,.76) (e){};

\filldraw[color=orange!1] (e.corner 1)--(e.corner 2)--(e.corner 12)--(e.corner 1);

\draw[thick] (e.corner 1)--(e.corner 2)--(e.corner 12)--(e.corner 1);

\filldraw[color=orange!15] (e.corner 2)--(e.corner 12)--(e.corner 4)--(e.corner 3)--(e.corner 2);
\draw[thick] (e.corner 2)--(e.corner 12)--(e.corner 4)--(e.corner 3)--(e.corner 2);

\filldraw[color=orange!35] (e.corner 4)--(e.corner 5)--(e.corner 12)--(e.corner 4);
\draw[thick] (e.corner 4)--(e.corner 5)--(e.corner 12)--(e.corner 4);

\filldraw[color=orange!70] (e.corner 6)--(e.corner 5)--(e.corner 12)--(e.corner 11)--(e.corner 7)--(e.corner 6);
\draw[thick] (e.corner 6)--(e.corner 5)--(e.corner 12)--(e.corner 11)--(e.corner 7)--(e.corner 6);

\filldraw[color=orange!100] (e.corner 7)--(e.corner 8)--(e.corner 9)--(e.corner 10)--(e.corner 11)--(e.corner 7);
\draw[thick] (e.corner 7)--(e.corner 8)--(e.corner 9)--(e.corner 10)--(e.corner 11)--(e.corner 7);

\node[regular polygon,regular polygon sides=12,minimum size=7cm,draw,rotate=30] at (.74,.76) (e){};

\foreach \x in {1,...,12}{\node[circle,fill,inner sep=1.5pt] at (e.corner \x) {};}

\node[above] at (e.corner 1) {1};
\node[above] at (e.corner 12) {12};
\foreach \x in {2,...,3}{\node[left] at (e.corner \x) {\x};}

\foreach \x in {4,...,7}{\node[below] at (e.corner \x) {\x};}

\foreach \x in {8,...,11}{\node[right] at (e.corner \x) {\x};}

{\draw (e.corner 2) -- (e.corner 12) {};}
{\draw (e.corner 2) -- (e.corner 4) {};}
{\draw (e.corner 4) -- (e.corner 12) {};}
{\draw (e.corner 5) -- (e.corner 12) {};}
{\draw (e.corner 5) -- (e.corner 7) {};}
{\draw (e.corner 6) -- (e.corner 7) {};}
{\draw (e.corner 7) -- (e.corner 10) {};}
{\draw (e.corner 7) -- (e.corner 9) {};}
{\draw (e.corner 5) -- (e.corner 11) {};}
{\draw (e.corner 7) -- (e.corner 11) {};}

\node at (0.1,4.3) {\textbf{1}};
\node at (-1,3) {\textbf{2}};
\node at (-1.3,.5) {\textbf{4}};
\node at (0.8,1.5) {\textbf{5}};
\node at (3.5,1.2) {\textbf{7}};

\filldraw [white] (e.corner 1) circle (1.25pt);
%\node[above] at (e.corner 1) {$x_1$};
\filldraw [white] (e.corner 2) circle (1.25pt);
%\node[above] at (e.corner 2) {$x_2$};
\filldraw [white] (e.corner 4) circle (1.25pt);
%\node[left] at (e.corner 4) {$x_3$};
\filldraw [white] (e.corner 5) circle (1.25pt);
%\node[left] at (e.corner 5) {$x_4$};
\filldraw [white] (e.corner 7) circle (1.25pt);
%\node[below] at (e.corner 7) {$x_5$};

\end{tikzpicture}

\end{center}

We now label the remaining triangles greedily in the order that fans appear in $\mathcal{F}(T)$.

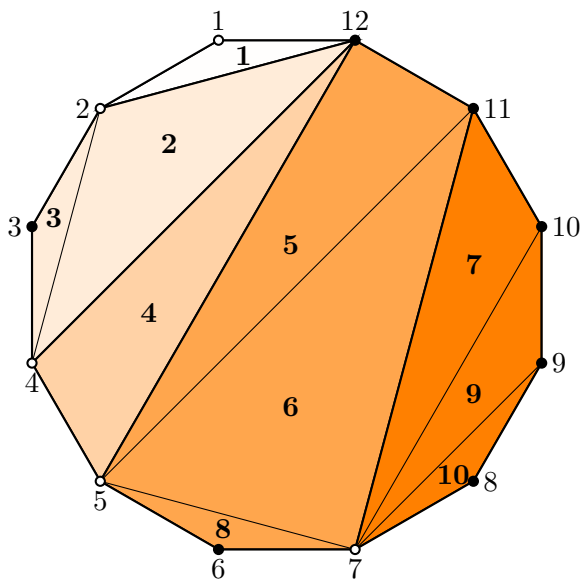
\begin{center}
\begin{tikzpicture}[scale=.9]
\node[regular polygon,regular polygon sides=12,minimum size=7cm,draw,rotate=30] at (.74,.76) (e){};

\filldraw[color=orange!1] (e.corner 1)--(e.corner 2)--(e.corner 12)--(e.corner 1);

\draw[thick] (e.corner 1)--(e.corner 2)--(e.corner 12)--(e.corner 1);

\filldraw[color=orange!15] (e.corner 2)--(e.corner 12)--(e.corner 4)--(e.corner 3)--(e.corner 2);
\draw[thick] (e.corner 2)--(e.corner 12)--(e.corner 4)--(e.corner 3)--(e.corner 2);

\filldraw[color=orange!35] (e.corner 4)--(e.corner 5)--(e.corner 12)--(e.corner 4);
\draw[thick] (e.corner 4)--(e.corner 5)--(e.corner 12)--(e.corner 4);

\filldraw[color=orange!70] (e.corner 6)--(e.corner 5)--(e.corner 12)--(e.corner 11)--(e.corner 7)--(e.corner 6);
\draw[thick] (e.corner 6)--(e.corner 5)--(e.corner 12)--(e.corner 11)--(e.corner 7)--(e.corner 6);

\filldraw[color=orange!100] (e.corner 7)--(e.corner 8)--(e.corner 9)--(e.corner 10)--(e.corner 11)--(e.corner 7);
\draw[thick] (e.corner 7)--(e.corner 8)--(e.corner 9)--(e.corner 10)--(e.corner 11)--(e.corner 7);

\node[regular polygon,regular polygon sides=12,minimum size=7cm,draw,rotate=30] at (.74,.76) (e){};

\foreach \x in {1,...,12}{\node[circle,fill,inner sep=1.5pt] at (e.corner \x) {};}

\node[above] at (e.corner 1) {1};
\node[above] at (e.corner 12) {12};
\foreach \x in {2,...,3}{\node[left] at (e.corner \x) {\x};}

\foreach \x in {4,...,7}{\node[below] at (e.corner \x) {\x};}

\foreach \x in {8,...,11}{\node[right] at (e.corner \x) {\x};}

{\draw (e.corner 2) -- (e.corner 12) {};}
{\draw (e.corner 2) -- (e.corner 4) {};}
{\draw (e.corner 4) -- (e.corner 12) {};}
{\draw (e.corner 5) -- (e.corner 12) {};}
{\draw (e.corner 5) -- (e.corner 7) {};}
{\draw (e.corner 6) -- (e.corner 7) {};}
{\draw (e.corner 7) -- (e.corner 10) {};}
{\draw (e.corner 7) -- (e.corner 9) {};}
{\draw (e.corner 5) -- (e.corner 11) {};}
{\draw (e.corner 7) -- (e.corner 11) {};}

\node at (0.1,4.3) {\textbf{1}};
\node at (-1,3) {\textbf{2}};
\node at (-2.7,1.9) {\textbf{3}};
\node at (-1.3,.5) {\textbf{4}};
\node at (0.8,1.5) {\textbf{5}};
\node at (0.8,-.9) {\textbf{6}};
\node at (-0.2,-2.7) {\textbf{8}};
\node at (3.5,1.2) {\textbf{7}};
\node at (3.5,-.7) {\textbf{9}};
\node at (3.2,-1.9) {\textbf{10}};

\filldraw [white] (e.corner 1) circle (1.25pt);
%\node[above] at (e.corner 1) {$x_1$};
\filldraw [white] (e.corner 2) circle (1.25pt);
%\node[above] at (e.corner 2) {$x_2$};
\filldraw [white] (e.corner 4) circle (1.25pt);
%\node[left] at (e.corner 4) {$x_3$};
\filldraw [white] (e.corner 5) circle (1.25pt);
%\node[left] at (e.corner 5) {$x_4$};
\filldraw [white] (e.corner 7) circle (1.25pt);
%\node[below] at (e.corner 7) {$x_5$};
\end{tikzpicture}

\end{center}

\editone{Hence, $A_1=\{1\}$, $A_2=\{2,3\}$, $A_3=\{4\}$, $A_4=\{5,6,8\}$, and $A_5=\{7,9,10\}$. This gives the following tableaux.}

\begin{center}

\begin{ytableau}
       \none & \none & 1 \\
       \none & \none & 4 \\
       2 &5 & 7 \\
       3&6&9\\
       8&\none&\none\\
              10&\none&\none\\
\end{ytableau}

\end{center}

\end{example}

\section{Dissections and Triangulations}\label{sec:enumeration}

In this section, we construct a combinatorial bijection between $\Dis(n+2,i)$ and $\Tri(n+i+1,i+1,0)$ as mentioned in the discussion following Corollary \ref{cor:disec_bij}. 
%
%\begin{corollary}
%The following are equal.
%\begin{enumerate}\label{cor:disec_bij}
%\item The number of dissections of an $(n+2)$-gon with $i$ chords.
%\item $\#\Skyt(n-i+1,i,2)$.
%\item $\#\Dyck(n+i+1,i+1,0)$. That is, the number of Dyck paths of semilength $n+i+1$ with $i+1$ long ascents and no singleton ascents.
%\item $\#\Tri(n+i+3,i+1,0)$. That is, the number of triangulations of an $(n+i+3)$-gon with $i+1$ non-singleton fans.
%\end{enumerate}
%\end{corollary}
%
%We provide a direct combinatorial bijective proof for the (1)-(3) and (1)-(4) equalities.

\editone{First, we define the map from $\Tri(n+i+3,i+1,0)$ to $\Dis(n+2,i)$.}
\begin{Def} \label{def:tri_to_dis}
Let $T$ be a triangulation of an $(n+i+3)$-gon with $i+1$ non-singular fans and no singular fans. \editone{We know $\mathcal{F}(T)$ is of the form $\mathcal{F}(T)=(F_1,\dots, F_{i+1})$}. Remove the internal diagonals of $F_j$ in $T$ for each $j$, leaving us with exactly $i$ diagonals in $T$. Let $x_j$ be the origin of $F_j$. For each vertex $x_j$, let $y_j$ be the immediate vertex that follows $x_j$ counterclockwise. Note that it is possible to have $y_j=x_{j+1}$. Also, we always have that the $y_j$ is a vertex in $F_j$, since $(x_j,y_j)$ must bound a triangle, and by definition this triangle is a part of $F_j$. 
Contract each edge $(x_j,y_j)$, creating an $(n+2)$-gon. Note that the vertex labeled 1 will always be the origin of $F_1$, so consequently we always contract $(1,2)$. Let 1 be the label of the new vertex after contracting this edge. Relabel the vertices in increasing counterclockwise order, starting at the original vertex 1.  
 Since no fan of $T$ was singular, it must be that the contractions preserved all $i$ diagonals, giving us a dissection in $\Dis(n+2,i)$.
\end{Def}

\editone{Now we define the map inverse to the one given above in definition \ref{def:tri_to_dis}, which is a map from $\Dis(n+2,i)$ to $\Tri(n+i+3,i+1,0)$.}

\begin{Def}\label{def:dis_to_tri}
For the reverse map, let $D$ be a dissection of an $(n+2)$-gon with $i$ chords, say $c_1,c_2,\dots, c_i$. We assume, as with triangulations, that the vertices of $D$ are already labeled with the numbers 1 through $n+2$.
%We will construct an ordered list $O$ of vertices. 
\editone{We will describe a process which allows us to add new vertices and edges to $D$.} Let $1'$ be a new vertex so that $(1',1)$ is an edge and $(1',2)$ is an edge. Delete the edge $(1,2)$. If $1$ was incident to more than one chords, shift all chords that do not form a triangle with the edge $(1,n+2)$ so that they are incident with $1'$ instead of $2$. 
Now, let $x$ be the next vertex counterclockwise to $1$ incident to a chord. (Note it may be that $x=1'$.) Proceed with the following procedure. 
%Let $x$ be the first vertex counterclockwise to $1$ incident to a chord. Follow the procedure below.
\begin{enumerate}
\item Let $c_{j_1},c_{j_2},\dots, c_{j_k}$ be the list of chords incident with $x$. 
%\item Append $x$ to the end of $O$.
\item Let $z$ be the vertex immediately counterclockwise of $x$. Remove the edge $(x,z)$ and add a new vertex $x'$ along with edges $(x,x')$ and $(x',z)$. 
\item If $x$ is adjacent to exactly one chord, continue to step (5). Otherwise, let $y$ be the vertex immediately clockwise to $x$. The edge $(y,x)$ bounds a closed region which contains exactly one chord $c_{j_\ell}$. For each $c_{j_m}$ with $m\neq \ell$, change its incidence with $x$ to an incidence with $x'$. 
\item If, after doing the prior step, we add a boundary edge to a region that we have already added a boundary edge to, undo the prior step and continue to the next step.
\item Move to the next vertex counterclockwise to $x$ incident to some chord, calling this new vertex $x$. (Note this new vertex may be the vertex $x'$ constructed in step (2).) If $x$ is a vertex we have already visited before, terminate the procedure. Otherwise, restart at step (1). 
\end{enumerate}
After doing this, observe that no region is a triangle. Also observe that we added a single edge to the boundary for each region (hence the importance of step (4)), and so we now have an $(n+i+3)$-gon. Relabel the vertices, starting at the vertex labeled 1 and continuing counterclockwise. We can decompose our new polygon into $i+1$ regions, labeled $P_1,P_2,\dots, P_{i+1}$. Make each of these fans so that the origin of $P_j$ is the vertex with the minimum label of $P_j$. This gives us a triangulation of an $(n+i+3)$-gon with $i+1$ non-singular fans and no singular fans.
\end{Def}

\begin{Rmk}
There are a couple of things to keep in mind that may help justify why the maps given in Definitions \ref{def:tri_to_dis} to \ref{def:dis_to_tri} are mutual inverses.
\begin{enumerate}
\item The edges we contract going from a triangulation to a dissection are exactly the edges we add back going from a dissection to a triangulation. This is because the origins of fans in triangulations are always chosen by the smallest vertex appearing in a fan, which appear sooner traveling counterclockwise around the polygons than vertices with larger labels. The regions in a dissection are ultimately what become our fans for a triangulation, so we consequently always add an edge on the boundary of a dissection right after the vertex that would end up being the origin for a fan.

\item The chords of a dissection should be viewed as the parts of the triangulation that ultimately form the boundaries of the fans (along with the actual boundary of the polygon). Hence, we can not expect two such chords to remain incident in the triangulation, as this would alter the number of fans. 
\end{enumerate}
\end{Rmk}

See Figure \ref{fig:tri_to_disec} below to see an illustration of the map from a triangulation to a dissection and Figure \ref{fig:dissec_to_tri} to see an illustration of the map of the other direction.

\newpage

\begin{figure}[h]

\begin{minipage}{.45\textwidth}
(a)
\begin{center}
\begin{tikzpicture}[scale=1]
\node[regular polygon,regular polygon sides=16,minimum size=5cm,draw,rotate=22.5] at (.74,.76) (e){};

\filldraw[color=orange!0] (e.corner 1)--(e.corner 16)--(e.corner 5)--(e.corner 4)--(e.corner 3)--(e.corner 2)--(e.corner 1);
\draw[thick] (e.corner 1)--(e.corner 16)--(e.corner 5)--(e.corner 4)--(e.corner 3)--(e.corner 2)--(e.corner 1);

\filldraw[color=orange!15] (e.corner 5)--(e.corner 16)--(e.corner 12)--(e.corner 6)--(e.corner 5);
\draw[thick] (e.corner 5)--(e.corner 16)--(e.corner 12)--(e.corner 6)--(e.corner 5);

\filldraw[color=orange!35] (e.corner 6)--(e.corner 7)--(e.corner 10)--(e.corner 11)--(e.corner 12)--(e.corner 6);
\draw[thick] (e.corner 6)--(e.corner 7)--(e.corner 10)--(e.corner 11)--(e.corner 12)--(e.corner 6);

\filldraw[color=orange!70] (e.corner 7)--(e.corner 8)--(e.corner 9)--(e.corner 10)--(e.corner 7);
\draw[thick] (e.corner 7)--(e.corner 8)--(e.corner 9)--(e.corner 10)--(e.corner 7);

\filldraw[color=orange!100] (e.corner 12)--(e.corner 13)--(e.corner 14)--(e.corner 15)--(e.corner 16)--(e.corner 12);
\draw[thick] (e.corner 12)--(e.corner 13)--(e.corner 14)--(e.corner 15)--(e.corner 16)--(e.corner 12);

\node[regular polygon,regular polygon sides=16,minimum size=5cm,draw,rotate=22.5] at (.74,.76) (e){};

\foreach \x in {1,...,16}{\node[circle,fill,inner sep=1.5pt] at (e.corner \x) {};}

\node[above] at (e.corner 1) {$1$};
\node[above] at (e.corner 2) {$2$};
\node[left] at (e.corner 3) {$3$};
\node[left] at (e.corner 4) {$4$};
\node[left] at (e.corner 5) {$5$};
\node[left] at (e.corner 6) {$6$};
\node[below] at (e.corner 7) {$7$};
\node[below] at (e.corner 8) {$8$};
\node[below] at (e.corner 9) {$9$};
\node[below] at (e.corner 10) {$10$};
\node[right] at (e.corner 11) {$11$};
\node[right] at (e.corner 12) {$12$};
\node[right] at (e.corner 13) {$13$};
\node[above] at (e.corner 14) {$14$};
\node[above] at (e.corner 15) {$15$};
\node[above] at (e.corner 16) {$16$};

{\draw (e.corner 1) -- (e.corner 3) {};}
{\draw (e.corner 1) -- (e.corner 5) {};}
{\draw (e.corner 1) -- (e.corner 4) {};}
{\draw (e.corner 1) -- (e.corner 16) {};}
{\draw (e.corner 5) -- (e.corner 16) {};}
{\draw (e.corner 5) -- (e.corner 12) {};}
{\draw (e.corner 7) -- (e.corner 10) {};}
{\draw (e.corner 9) -- (e.corner 7) {};}
{\draw (e.corner 6) -- (e.corner 12) {};}
{\draw (e.corner 6) -- (e.corner 11) {};}
{\draw (e.corner 6) -- (e.corner 10) {};}
{\draw (e.corner 13) -- (e.corner 12) {};}
{\draw (e.corner 14) -- (e.corner 12) {};}
{\draw (e.corner 16) -- (e.corner 12) {};}
{\draw (e.corner 15) -- (e.corner 12) {};}

\filldraw [white] (e.corner 1) circle (1.25pt);
\filldraw [white] (e.corner 5) circle (1.25pt);
\filldraw [white] (e.corner 6) circle (1.25pt);
\filldraw [white] (e.corner 7) circle (1.25pt);
\filldraw [white] (e.corner 12) circle (1.25pt);
\end{tikzpicture}
\end{center}
\end{minipage}\hspace{.5in}
\begin{minipage}{.45\textwidth}
(b)
\begin{center}
\begin{tikzpicture}[scale=1]
\node[regular polygon,regular polygon sides=16,minimum size=5cm,draw,rotate=22.5] at (.74,.76) (e){};

\node[regular polygon,regular polygon sides=16,minimum size=5cm,draw,rotate=22.5] at (.74,.76) (e){};

\filldraw[color=orange!0] (e.corner 1)--(e.corner 16)--(e.corner 5)--(e.corner 4)--(e.corner 3)--(e.corner 2)--(e.corner 1);
\draw[thick] (e.corner 1)--(e.corner 16)--(e.corner 5)--(e.corner 4)--(e.corner 3)--(e.corner 2)--(e.corner 1);

\filldraw[color=orange!15] (e.corner 5)--(e.corner 16)--(e.corner 12)--(e.corner 6)--(e.corner 5);
\draw[thick] (e.corner 5)--(e.corner 16)--(e.corner 12)--(e.corner 6)--(e.corner 5);

\filldraw[color=orange!35] (e.corner 6)--(e.corner 7)--(e.corner 10)--(e.corner 11)--(e.corner 12)--(e.corner 6);
\draw[thick] (e.corner 6)--(e.corner 7)--(e.corner 10)--(e.corner 11)--(e.corner 12)--(e.corner 6);

\filldraw[color=orange!70] (e.corner 7)--(e.corner 8)--(e.corner 9)--(e.corner 10)--(e.corner 7);
\draw[thick] (e.corner 7)--(e.corner 8)--(e.corner 9)--(e.corner 10)--(e.corner 7);

\filldraw[color=orange!100] (e.corner 12)--(e.corner 13)--(e.corner 14)--(e.corner 15)--(e.corner 16)--(e.corner 12);
\draw[thick] (e.corner 12)--(e.corner 13)--(e.corner 14)--(e.corner 15)--(e.corner 16)--(e.corner 12);

\node[regular polygon,regular polygon sides=16,minimum size=5cm,draw,rotate=22.5] at (.74,.76) (e){};

\foreach \x in {1,...,16}{\node[circle,fill,inner sep=1.5pt] at (e.corner \x) {};}

\node[above] at (e.corner 1) {$1$};
\node[above] at (e.corner 2) {$2$};
\node[left] at (e.corner 3) {$3$};
\node[left] at (e.corner 4) {$4$};
\node[left] at (e.corner 5) {$5$};
\node[left] at (e.corner 6) {$6$};
\node[below] at (e.corner 7) {$7$};
\node[below] at (e.corner 8) {$8$};
\node[below] at (e.corner 9) {$9$};
\node[below] at (e.corner 10) {$10$};
\node[right] at (e.corner 11) {$11$};
\node[right] at (e.corner 12) {$12$};
\node[right] at (e.corner 13) {$13$};
\node[above] at (e.corner 14) {$14$};
\node[above] at (e.corner 15) {$15$};
\node[above] at (e.corner 16) {$16$};

{\draw (e.corner 5) -- (e.corner 16) {};}
{\draw (e.corner 6) -- (e.corner 12) {};}
{\draw (e.corner 7) -- (e.corner 10) {};}
{\draw (e.corner 16) -- (e.corner 12) {};}

\filldraw [white] (e.corner 1) circle (1.25pt);
\filldraw [white] (e.corner 5) circle (1.25pt);
\filldraw [white] (e.corner 6) circle (1.25pt);
\filldraw [white] (e.corner 7) circle (1.25pt);
\filldraw [white] (e.corner 12) circle (1.25pt);
\end{tikzpicture}
\end{center}
\end{minipage}

\vspace{.5in}

\begin{minipage}{.45\textwidth}
(c)
\begin{center}
\begin{tikzpicture}[scale=.3]
\node[regular polygon,regular polygon sides=16,minimum size=5cm,draw,rotate=22.5] at (.74,.76) (e){};

%\filldraw[color=black!0] (e.corner 1)--(e.corner 15)--(e.corner 5)--(e.corner 4)--(e.corner 3)--(e.corner 2)--(e.corner 1);
%
%\filldraw[color=black!20] (e.corner 5)--(e.corner 16)--(e.corner 11)--(e.corner 6)--(e.corner 5);
%
%
%\filldraw[color=black!40] (e.corner 6)--(e.corner 7)--(e.corner 8)--(e.corner 9)--(e.corner 10)--(e.corner 11)--(e.corner 6);
%
%\filldraw[color=black!60] (e.corner 11)--(e.corner 12)--(e.corner 13)--(e.corner 14)--(e.corner 15)--(e.corner 16)--(e.corner 11);

\node[regular polygon,regular polygon sides=16,minimum size=5cm,draw,rotate=22.5] at (.74,.76) (e){};

{\draw[thick,red] (e.corner 2) -- (e.corner 1) {};}
{\draw[thick,red] (e.corner 6) -- (e.corner 5) {};}
{\draw[thick,red] (e.corner 6) -- (e.corner 7) {};}
{\draw[thick,red] (e.corner 8) -- (e.corner 7) {};}
{\draw[thick,red] (e.corner 13) -- (e.corner 12) {};}

\foreach \x in {1,...,16}{\node[circle,fill,inner sep=1.5pt] at (e.corner \x) {};}

\node[above] at (e.corner 1) {$1$};
\node[above] at (e.corner 2) {$2$};
\node[left] at (e.corner 3) {$3$};
\node[left] at (e.corner 4) {$4$};
\node[left] at (e.corner 5) {$5$};
\node[left] at (e.corner 6) {$6$};
\node[below] at (e.corner 7) {$7$};
\node[below] at (e.corner 8) {$8$};
\node[below] at (e.corner 9) {$9$};
\node[below] at (e.corner 10) {$10$};
\node[right] at (e.corner 11) {$11$};
\node[right] at (e.corner 12) {$12$};
\node[right] at (e.corner 13) {$13$};
\node[above] at (e.corner 14) {$14$};
\node[above] at (e.corner 15) {$15$};
\node[above] at (e.corner 16) {$16$};

{\draw (e.corner 5) -- (e.corner 16) {};}
{\draw (e.corner 6) -- (e.corner 12) {};}
{\draw (e.corner 7) -- (e.corner 10) {};}
{\draw (e.corner 16) -- (e.corner 12) {};}
\end{tikzpicture}
\end{center}
\end{minipage}\hspace{.5in}
\begin{minipage}{.45\textwidth}
(d)
\begin{center}
\begin{tikzpicture}[scale=.3]
\node[regular polygon,regular polygon sides=11,minimum size=5cm,draw,rotate=30] at (.74,.76) (e){};

%\filldraw[color=black!0] (e.corner 1)--(e.corner 12)--(e.corner 5)--(e.corner 4)--(e.corner 3)--(e.corner 2)--(e.corner 1);
%
%\filldraw[color=black!20] (e.corner 5)--(e.corner 12)--(e.corner 9)--(e.corner 5);
%
%
%\filldraw[color=black!40] (e.corner 5)--(e.corner 6)--(e.corner 7)--(e.corner 8)--(e.corner 9)--(e.corner 5);
%
%\filldraw[color=black!60] (e.corner 9)--(e.corner 10)--(e.corner 11)--(e.corner 12)--(e.corner 9);

\node[regular polygon,regular polygon sides=11,minimum size=5cm,draw,rotate=30] at (.74,.76) (e){};

\foreach \x in {1,...,11}{\node[circle,fill,inner sep=1.5pt] at (e.corner \x) {};}

\node[above] at (e.corner 1) {$1$};
\node[left] at (e.corner 2) {$2$};
\node[left] at (e.corner 3) {$3$};
\node[left] at (e.corner 4) {$4$};
\node[below] at (e.corner 5) {$5$};
\node[below] at (e.corner 6) {$6$};
\node[right] at (e.corner 7) {$7$};
\node[right] at (e.corner 8) {$8$};
\node[right] at (e.corner 9) {$9$};
\node[above] at (e.corner 10) {$10$};
\node[above] at (e.corner 11) {$11$};

{\draw (e.corner 4) -- (e.corner 11) {};}
{\draw (e.corner 4) -- (e.corner 8) {};}
{\draw (e.corner 4) -- (e.corner 6) {};}
{\draw (e.corner 11) -- (e.corner 8) {};}
\end{tikzpicture}
\end{center}
\end{minipage}
\caption{The steps transforming an element of $\Tri(16,5,0)$, Figure \ref{fig:tri_to_disec}(a), to a dissection of a 11-gon with 4 chords, Figure \ref{fig:tri_to_disec}(d). The fans of (a) are colored different shades of gray, and we omit this shading after the fans no longer become relevant in part (c). In (c), the red boundary edges (1,2), (5,6), (6,7), (7,8), and (12,13) are the boundary edges that get contracted. The vertex labeled $1$ in (d) is the vertex labeled $2$ in the other parts.}
\label{fig:tri_to_disec}
\end{figure}

\newpage

\begin{figure}[h]

\begin{minipage}{.45\textwidth}
(a)
\begin{center}

\begin{tikzpicture}[scale=.3]
\node[regular polygon,regular polygon sides=11,minimum size=5cm,draw,rotate=16] at (.74,.76) (e){};

%\filldraw[color=black!0] (e.corner 1)--(e.corner 12)--(e.corner 5)--(e.corner 4)--(e.corner 3)--(e.corner 2)--(e.corner 1);
%
%\filldraw[color=black!20] (e.corner 5)--(e.corner 12)--(e.corner 9)--(e.corner 5);
%
%
%\filldraw[color=black!40] (e.corner 5)--(e.corner 6)--(e.corner 7)--(e.corner 8)--(e.corner 9)--(e.corner 5);
%
%\filldraw[color=black!60] (e.corner 9)--(e.corner 10)--(e.corner 11)--(e.corner 12)--(e.corner 9);

\node[regular polygon,regular polygon sides=11,minimum size=5cm,draw,rotate=16] at (.74,.76) (e){};

\foreach \x in {1,...,11}{\node[circle,fill,inner sep=1.5pt] at (e.corner \x) {};}

\node[above] at (e.corner 1) {$1$};
\node[left] at (e.corner 2) {$2$};
\node[left] at (e.corner 3) {$3$};
\node[left] at (e.corner 4) {$4$};
\node[below] at (e.corner 5) {$5$};
\node[below] at (e.corner 6) {$6$};
\node[right] at (e.corner 7) {$7$};
\node[right] at (e.corner 8) {$8$};
\node[right] at (e.corner 9) {$9$};
\node[above] at (e.corner 10) {$10$};
\node[above] at (e.corner 11) {$11$};

{\draw (e.corner 4) -- (e.corner 11) {};}
{\draw (e.corner 4) -- (e.corner 8) {};}
{\draw (e.corner 4) -- (e.corner 6) {};}
{\draw (e.corner 11) -- (e.corner 8) {};}
\end{tikzpicture}
\end{center}
\end{minipage}\hspace{.5in}
\begin{minipage}{.45\textwidth}
(b)
\begin{center}

\begin{tikzpicture}[scale=.3]
\node[regular polygon,regular polygon sides=13,minimum size=5cm,draw,rotate=15] at (.74,.76) (e){};

%\filldraw[color=black!0] (e.corner 1)--(e.corner 12)--(e.corner 5)--(e.corner 4)--(e.corner 3)--(e.corner 2)--(e.corner 1);
%
%\filldraw[color=black!20] (e.corner 5)--(e.corner 12)--(e.corner 9)--(e.corner 5);
%
%
%\filldraw[color=black!40] (e.corner 5)--(e.corner 6)--(e.corner 7)--(e.corner 8)--(e.corner 9)--(e.corner 5);
%
%\filldraw[color=black!60] (e.corner 9)--(e.corner 10)--(e.corner 11)--(e.corner 12)--(e.corner 9);

\node[regular polygon,regular polygon sides=13,minimum size=5cm,draw,rotate=15] at (.74,.76) (e){};

\foreach \x in {1,...,13}{\node[circle,fill,inner sep=1.5pt] at (e.corner \x) {};}

\node[above] at (e.corner 1) {$1$};
\node[above] at (e.corner 2) {$1'$};
\node[left] at (e.corner 3) {$2$};
\node[left] at (e.corner 4) {$3$};
\node[left] at (e.corner 5) {$4$};
\node[below] at (e.corner 6) {$4'$};
\node[below] at (e.corner 7) {$5$};
\node[below] at (e.corner 8) {$6$};
\node[right] at (e.corner 9) {$7$};
\node[right] at (e.corner 10) {$8$};
\node[right] at (e.corner 11) {$9$};
\node[above] at (e.corner 12) {$10$};
\node[above] at (e.corner 13) {$11$};

{\draw (e.corner 5) -- (e.corner 13) {};}
{\draw (e.corner 6) -- (e.corner 10) {};}
{\draw (e.corner 6) -- (e.corner 8) {};}
{\draw (e.corner 13) -- (e.corner 10) {};}
\end{tikzpicture}
\end{center}
\end{minipage}

\vspace{.5in}

\begin{minipage}{.45\textwidth}
(c)

\begin{center}
\begin{tikzpicture}[scale=.3]
\node[regular polygon,regular polygon sides=16,minimum size=5cm,draw,rotate=22.5] at (.74,.76) (e){};

%\filldraw[color=black!0] (e.corner 1)--(e.corner 12)--(e.corner 5)--(e.corner 4)--(e.corner 3)--(e.corner 2)--(e.corner 1);
%
%\filldraw[color=black!20] (e.corner 5)--(e.corner 12)--(e.corner 9)--(e.corner 5);
%
%
%\filldraw[color=black!40] (e.corner 5)--(e.corner 6)--(e.corner 7)--(e.corner 8)--(e.corner 9)--(e.corner 5);
%
%\filldraw[color=black!60] (e.corner 9)--(e.corner 10)--(e.corner 11)--(e.corner 12)--(e.corner 9);

\node[regular polygon,regular polygon sides=16,minimum size=5cm,draw,rotate=22.5] at (.74,.76) (e){};

\foreach \x in {1,...,16}{\node[circle,fill,inner sep=1.5pt] at (e.corner \x) {};}

\node[above] at (e.corner 1) {$1$};
\node[above] at (e.corner 2) {$1'$};
\node[left] at (e.corner 3) {$2$};
\node[left] at (e.corner 4) {$3$};
\node[left] at (e.corner 5) {$4$};
\node[below] at (e.corner 6) {$4'$};
\node[below] at (e.corner 7) {$4''$};
\node[below] at (e.corner 8) {$4'''$};
\node[below] at (e.corner 9) {$5$};
\node[below] at (e.corner 10) {$6$};
\node[right] at (e.corner 11) {$7$};
\node[right] at (e.corner 12) {$8$};
\node[right] at (e.corner 13) {$8'$};
\node[right] at (e.corner 14) {$9$};
\node[above] at (e.corner 15) {$10$};
\node[above] at (e.corner 16) {$11$};

{\draw (e.corner 5) -- (e.corner 16) {};}
{\draw (e.corner 6) -- (e.corner 12) {};}
{\draw (e.corner 7) -- (e.corner 10) {};}
{\draw (e.corner 16) -- (e.corner 12) {};}
\end{tikzpicture}
\end{center}
\end{minipage}\hspace{.5in}
\begin{minipage}{.45\textwidth}
(d)

\begin{center}
\begin{tikzpicture}[scale=1	]
\node[regular polygon,regular polygon sides=16,minimum size=5cm,draw,rotate=22.5] at (.74,.76) (e){};

\filldraw[color=orange!0] (e.corner 1)--(e.corner 16)--(e.corner 5)--(e.corner 4)--(e.corner 3)--(e.corner 2)--(e.corner 1);
\draw[thick] (e.corner 1)--(e.corner 16)--(e.corner 5)--(e.corner 4)--(e.corner 3)--(e.corner 2)--(e.corner 1);

\filldraw[color=orange!15] (e.corner 5)--(e.corner 16)--(e.corner 12)--(e.corner 6)--(e.corner 5);
\draw[thick] (e.corner 5)--(e.corner 16)--(e.corner 12)--(e.corner 6)--(e.corner 5);

\filldraw[color=orange!35] (e.corner 6)--(e.corner 7)--(e.corner 10)--(e.corner 11)--(e.corner 12)--(e.corner 6);
\draw[thick] (e.corner 6)--(e.corner 7)--(e.corner 10)--(e.corner 11)--(e.corner 12)--(e.corner 6);

\filldraw[color=orange!70] (e.corner 7)--(e.corner 8)--(e.corner 9)--(e.corner 10)--(e.corner 7);
\draw[thick] (e.corner 7)--(e.corner 8)--(e.corner 9)--(e.corner 10)--(e.corner 7);

\filldraw[color=orange!100] (e.corner 12)--(e.corner 13)--(e.corner 14)--(e.corner 15)--(e.corner 16)--(e.corner 12);
\draw[thick] (e.corner 12)--(e.corner 13)--(e.corner 14)--(e.corner 15)--(e.corner 16)--(e.corner 12);

\node[regular polygon,regular polygon sides=16,minimum size=5cm,draw,rotate=22.5] at (.74,.76) (e){};

\foreach \x in {1,...,16}{\node[circle,fill,inner sep=1.5pt] at (e.corner \x) {};}

\node[above] at (e.corner 1) {$1$};
\node[above] at (e.corner 2) {$2$};
\node[left] at (e.corner 3) {$3$};
\node[left] at (e.corner 4) {$4$};
\node[left] at (e.corner 5) {$5$};
\node[left] at (e.corner 6) {$6$};
\node[below] at (e.corner 7) {$7$};
\node[below] at (e.corner 8) {$8$};
\node[below] at (e.corner 9) {$9$};
\node[below] at (e.corner 10) {$10$};
\node[right] at (e.corner 11) {$11$};
\node[right] at (e.corner 12) {$12$};
\node[right] at (e.corner 13) {$13$};
\node[above] at (e.corner 14) {$14$};
\node[above] at (e.corner 15) {$15$};
\node[above] at (e.corner 16) {$16$};

{\draw (e.corner 1) -- (e.corner 3) {};}
{\draw (e.corner 1) -- (e.corner 5) {};}
{\draw (e.corner 1) -- (e.corner 4) {};}
{\draw (e.corner 1) -- (e.corner 16) {};}
{\draw (e.corner 5) -- (e.corner 16) {};}
{\draw (e.corner 5) -- (e.corner 12) {};}
{\draw (e.corner 7) -- (e.corner 10) {};}
{\draw (e.corner 9) -- (e.corner 7) {};}
{\draw (e.corner 6) -- (e.corner 12) {};}
{\draw (e.corner 6) -- (e.corner 11) {};}
{\draw (e.corner 6) -- (e.corner 10) {};}
{\draw (e.corner 13) -- (e.corner 12) {};}
{\draw (e.corner 14) -- (e.corner 12) {};}
{\draw (e.corner 16) -- (e.corner 12) {};}
{\draw (e.corner 15) -- (e.corner 12) {};}

\filldraw [white] (e.corner 1) circle (1.25pt);
\filldraw [white] (e.corner 5) circle (1.25pt);
\filldraw [white] (e.corner 6) circle (1.25pt);
\filldraw [white] (e.corner 7) circle (1.25pt);
\filldraw [white] (e.corner 12) circle (1.25pt);
\end{tikzpicture}
\end{center}
\end{minipage}
\caption{The steps transforming an a dissection of a 11-gon with 4 chords, Figure \ref{fig:dissec_to_tri}(a), to an element of $\Tri(16,5,0)$, Figure \ref{fig:dissec_to_tri}(d). In (b), we see the initial step of adding vertex $1'$ and an application of step (3) from the description of the combinatorial bijection for Corollary \ref{cor:disec_bij} (1)-(4), requiring us to change an adjacency of a chord. In (c), we then do this a second time to vertex $4'$ (as it is still adjacent to multiple chords) and add the remaining vertices as described in the map. Adding chords to the vertex with minimum label in each part gives us the triangulation in (d).}
\label{fig:dissec_to_tri}
\end{figure}

\bibliographystyle{amsalpha}
% or use \bibliographystyle{abbrv}
\bibliography{./George_SuJi_bib}

\end{document}